\documentclass[11pt,a4paper]{amsart}
\usepackage[foot]{amsaddr}

\makeatletter
\usepackage[english]{babel}
\usepackage[utf8]{inputenc}
\usepackage{csquotes}
\usepackage[T1]{fontenc}
\usepackage{latexsym}
\usepackage[leqno]{amsmath}
\usepackage{amssymb,amsthm,amsfonts}
\usepackage{mathtools}
\usepackage{mathrsfs}
\usepackage{dsfont}
\usepackage{stmaryrd}
\usepackage{comment}
\usepackage{todonotes}

\usepackage[normalem]{ulem}
\usepackage{cancel}

\usepackage{geometry}
\usepackage{nicefrac}

\usepackage{enumerate}
\usepackage[shortlabels]{enumitem}
\setlist[enumerate]{label=(\alph*),font=\normalshape}
\setlist[itemize]{font=\normalshape}
\let\originalitem\item
\renewcommand{\item}[1][]{%
	\if\relax\detokenize{#1}\relax%
		\originalitem%
	\else%
		\originalitem[#1]%
		\phantomsection
		\def\@currentlabel{#1}
	\fi%
}

\usepackage{hyperref}

\usepackage[noabbrev,capitalize]{cleveref}

\usepackage[style=numeric-comp]{biblatex}

\let\originalleft\left
\let\originalright\right
\renewcommand{\left}{\mathopen{}\mathclose\bgroup\originalleft}
\renewcommand{\right}{\aftergroup\egroup\originalright}

\allowdisplaybreaks
\theoremstyle{plain}
\newtheorem{theorem}{Theorem}[section]
\newtheorem{definition}[theorem]{Definition}
 
\newtheorem{lemma}[theorem]{Lemma} 
\newtheorem{corollary}[theorem]{Corollary}

\theoremstyle{definition}
\newtheorem{remark}[theorem]{Remark}

\AddToHook{env/definition/begin}{\crefalias{theorem}{definition}}
\AddToHook{env/proposition/begin}{\crefalias{theorem}{proposition}}
\AddToHook{env/lemma/begin}{\crefalias{theorem}{lemma}}
\AddToHook{env/corollary/begin}{\crefalias{theorem}{corollary}}
\AddToHook{env/notation/begin}{\crefalias{theorem}{notation}}
\AddToHook{env/remark/begin}{\crefalias{theorem}{remark}}

\newcommand{\R}{\mathbb{R}} 
 
\newcommand{\Z}{\mathbb{Z}}

\newcommand{\N}{\mathbb{N}}

\newcommand{\T}{\mathbb{T}}

\DeclareMathAlphabet{\othermathbb}{U}{bbold}{m}{n}

\newcommand{\der}[2][]{\@ifnextchar\der{\,#1\mathrm{d}#2\!}{\,#1\mathrm{d}#2}}

\DeclareMathOperator{\range}{range}

\DeclareMathOperator*{\essinf}{ess~inf}

\DeclareMathOperator{\Id}{Id}

\newcommand{\ee}{\mathrm{e}}
\newcommand{\ii}{\mathrm{i}}
\newcommand{\const}{\mathrm{const.}}

\newcommand{\set}[1]{{\left\{ #1 \right\}}}

\newcommand{\norm}[1]{\left\lVert #1 \right\rVert}

\newcommand{\dv}[3][]{%
	\if\relax\detokenize{#1}\relax%
		\frac{\mathrm{d} #2}{\mathrm{d} #3}%
	\else%
		\frac{\mathrm{d}^{#1} #2}{\mathrm{d} {#3}^{#1}}%
	\fi%
}

\newcommand{\dtsquare}{\dv[2]{}{t}}

\newcommand{\calM}{\mathcal{M}}

\newcommand{\dt}{\ \mathrm{d}t}
\makeatother

\bibliography{bibliography}

\begin{document}
	\title[Homoclinics, rogue waves and breathers in nonlinear lattice wave equations]{Homoclinics, rogue waves and breathers in nonlinear lattice wave equations}

	\author{Julia Henninger$^1$}
	\email{julia.henninger@kit.edu}
	\author{Wolfgang Reichel$^1$}
	\email{wolfgang.reichel@kit.edu}
	\address{$^1$Institute for Analysis, Karlsruhe Institute of Technology (KIT), D-76128 Karlsruhe, Germany}
	
	\date{\today}
	\subjclass[2020]{Primary: 35L71, 37C29; Secondary: 35A15, 49J35, }
    % \subjclass[2000]{Primary: 35L71, 49J35; Secondary: 35B10, 34L05}
    % 	35L71  Second-order semilinear hyperbolic equations
    % 49J35   Existence of solutions for minimax problems
    % 37C29   	Homoclinic and heteroclinic orbits for dynamical systems
    % 35A15   	Variational methods applied to PDEs
   
	\keywords{Nonlinear lattice equations, FPUT, rogue waves, homoclinic solutions, breathers, variational methods}

	\begin{abstract}
		We prove the existence of rogue waves and breathers in nonlinear lattice wave equations including the nonlinear Klein--Gordon and the FPUT lattice (with added local forces). The main feature of our lattice wave equation is that the kinetic part $\frac{\mathrm{d}^2}{\mathrm{d}t^2}$ is multiplied by a non-constant function $\frac{1}{d(t)}$ such that gaps in the spectrum of $\frac{1}{d(t)}\frac{\mathrm{d}^2}{\mathrm{d}t^2}$ open wide enough to include the spectrum of the spatial linear operator. We find solutions as critical points of an indefinite functional using a saddle-point method combined with concentration-compactness arguments. In case of temporally $T$-periodic coefficients and under identical assumptions as for rogue waves, the same variational method also provides existence of breather solutions whose temporal period is an arbitrary prescribed integer multiple of $T$. 
	\end{abstract}

	\maketitle
%\tableofcontents
	\section{Introduction and main result}

We are interested in particular solutions for the semilinear Klein--Gordon-type lattice equation
\begin{align}\label{eq:main_0}
    \frac{1}{d(t)}\ddot{u}_n  - \Delta_1 u_n + q_n u_n = g(u_n) \quad \mbox{ on } \R \times \Z
\end{align}
as well as for the FPUT(Fermi--Pasta--Ulam--Tsingou)-type equation \cite{fput}
\begin{align}\label{eq:main_00}
    \frac{1}{d(t)}\ddot{u}_n  - \Delta_1 u_n + q_n u_n = g\bigl(u_{n+1}-u_n\bigr)-g\bigl(u_n-u_{n-1}\bigr) \quad \mbox{ on } \R \times \Z
\end{align}
where $\Delta_1 u_n \coloneqq u_{n+1}-2 u_n+u_{n-1}$ denotes the discrete Laplacian on the lattice $\Z$. For the moment we understand \eqref{eq:main_0} and \eqref{eq:main_00} pointwise for every $(n,t)\in \Z\times \R$ with classical derivatives $\dot{u}_n(t) \coloneqq \frac{\der}{\der t}u_n(t)$, $\ddot{u}_n(t) \coloneqq \frac{\der^2}{\der t^2}u_n(t)$. Our goal is to construct rogue wave solutions, i.e., real-valued, nontrivial solutions which are localized in space and in time in the sense $\lim_{|n|+|t|\to\infty} u_n(t)=0$.

More generally, we investigate a framework which allows to simultaneously treat \eqref{eq:main_0} and \eqref{eq:main_00}. For this purpose we consider the nonlinear lattice wave equation
\begin{align}\label{eq:main}
    \frac{1}{d(t)}\ddot{u}  +Mu= B^* f^\sharp(t, Bu) \quad \mbox{ on } \R\times \Z
\end{align}
with bounded linear operators $M, B: \ell^2(\Z)\to \ell^2(\Z)$ and where we will see that the Nemytskii operator $u \mapsto f^\sharp(t, u) \coloneqq (f(n,t,u_n))_{n\in \Z}$ maps $\ell^2(\Z)$ to $\ell^2(\Z)$ for almost all $t\in \R$. Clearly, $M=-\Delta_1+q$, and $B=\Id$ and $f(n,t,s)=g(s)$ corresponds to \eqref{eq:main_0} whereas the \emph{difference operator} $B:\ell^2(\Z) \to \ell^2(\Z)$ given by $(Bu)_n=u_{n+1}-u_n$ for $n \in \Z$ and $f(n,t,s)=-g(s)$ corresponds to \eqref{eq:main_00}. We shall prove in our main result of Theorem~\ref{thm:main} the existence of a rogue wave $u$ of \eqref{eq:main} as a function $u\in H^2(\R;\ell^2(\Z))$ so that in particular $u(t)\in \ell^2(\Z)$ for almost all $t$. Therefore, one can also view rogue waves as homoclinic solutions $t\mapsto u(t)\in \ell^2(\Z)$ with $\lim_{|t|\to\infty} u(t)=0$. While the main emphasis of our paper is on rogue waves, it turns out that the existence of time-periodic solutions in $\ell^2(\Z)$ (breathers) can be obtained by the same variational method. More precisely, in those cases where where \eqref{eq:main} has an underlying $T$-periodic structure in time, we get the existence of breathers with periods $kT$ for arbitrary $k\in \N$, cf. Corollary~\ref{cor:breathers}, under exactly the same assumptions as for rogue waves in Theorem~\ref{thm:main}.

In the present work, we will show existence of rogue wave and breather solutions to \eqref{eq:main} by means of variational methods. Our approach is inspired by \cite{henninger} which shows existence of rogue wave solutions to a (spatially) continuous version of \eqref{eq:main}.

Next we discuss the assumptions on $M, B, d$. 
\begin{enumerate}[label=(L\arabic*)]
    \item $M: \ell^2(\Z) \to \ell^2(\Z)$ is linear, bounded, self-adjoint, positive, \label{as:M}
    \item $d \in L^\infty(\R)$ with $\essinf_\R d >0 $. \label{as:dbdd}
\end{enumerate}
With \ref{as:dbdd} we consider $\frac{1}{d(t)}\frac{\der^2}{\der t^2}: H^2(\R)\subset L^2_d(\R)\to L^2_d(\R)$ as a self-adjoint operator on the weighted $L^2$-space $L^2_d(\R)$. Further, we assume
\begin{enumerate}[label=(L\arabic*)] \setcounter{enumi}{2}
    \item $\sigma\left( M \right) \cap \sigma\left( -\frac{1}{d(t)} \dtsquare \right)=\emptyset$, \label{as:specgap}
    \item $B:\ell^\infty(\Z) \to \ell^\infty(\Z)$ is linear, bounded, and  $B|_{\ell^2(\Z)}$ is injective with $\range(B|_{\ell^2(\Z)}) \subset \ell^2(\Z)$. \label{as:B}
\end{enumerate}
For the nonlinearity $f$ and its primitive $F(n,t,s)=\int_0^s f(n,t,\sigma) \der \sigma$ we make the following assumptions:
\begin{enumerate}[label=(f\arabic*)]
    \item $f \in C(\Z \times \R \times \R)$, \label{as:f1}
    \item $f(n,t,s)=o(s)$ uniformly in $t$ and $n$, as $s \to 0$, \label{as:f2} 
    \item $s \mapsto \frac{f(n,t,s)}{|s|}$ strictly increasing on $(-\infty,0)$ and $(0,\infty)$ for all $(n,t)\in \Z\times\R$, \label{as:f3}
    \item $\frac{F(n,t,s)}{s^2} \to \infty$ uniformly in $t$ and $n$, as $|s| \to \infty$, \label{as:f4}
    \item $\exists C>0, p\in (1,\infty)$ such that $|f(n,t,s)| \leq C (1+|s|^p)$ for all $(n,t,s)\in \Z\times \R^2$. \label{as:f5}
\end{enumerate}
\begin{remark} \label{rem:comments} Let us make some comments on the assumptions. 
\begin{itemize} 
\item[(i)] If $M=-\Delta_1+q$ is the discrete Laplacian with an added potential $q=(q_n)_{n\in\Z}\in \ell^\infty(\Z)$, $q\geq 0$ as in \eqref{eq:main_0} and \eqref{eq:main_00} then it satisfies \ref{as:M} as we shall see in Lemma~\ref{lem:spec_disc_lap_q} in Appendix~\ref{sec:examples}. A further example is $M=\Delta^2_1+q$, $q\geq 0$, where $\Delta_1^2 u_n \coloneqq u_{n+2}-4u_{n+1}+6 u_n -4 u_{n-1}+u_{n-2}$ is a discretization of the bi-Laplacian.
\item[(ii)] Condition \ref{as:specgap} means that $0$ is not in the spectrum of the wave operator $L \coloneqq \frac{1}{d(t)}\frac{\mathrm{d}^2}{\mathrm{d}t^2}+M$ provided we consider it as a self-adjoint operator on a suitable function space and extend both $\frac{1}{d(t)}\frac{\der^2}{\der t^2}$ and $M$ to functions of $n$ and $t$, cf. Section~\ref{sec:varset}. This is central for the applicability of the variational method. A class of examples of operators $M$ and periodic two-step potentials $d$ such that condition \ref{as:specgap} is fulfilled are provided in Corollary~\ref{cor:zitierbar}.
\item[(iii)] Condition \ref{as:B} is satisfied both for $B=\Id$ and the difference operator $(Bu)_{n} = u_{n+1}-u_n$.
\item[(iv)] A prototypical example for the nonlinearity is $f(n,t,s)=\Gamma(n,t) |s|^{p-1}s$ with $\Gamma \in C(\Z\times \R)$ chosen according to one of four cases of assumptions given in Theorem~\ref{thm:main}. Clearly, one can also consider \eqref{eq:main} with a right hand side being a finite sum of terms of the form $B^* f^\sharp(t, Bu)$.
\end{itemize}
\end{remark}
Lastly, for given $N \in \Z$  we introduce the \emph{shift operator} $S_N : \ell^2(\Z) \to \ell^2(\Z)$ by $(S_N v)_n=v_{n+N}$. We say that $M$ and $B$ commute with $S_N$ if $MS_N=S_NM$ and $BS_N=S_NB$. With this at hand we can now state our main result. 

\begin{theorem}\label{thm:main}
    Assume that \ref{as:M}--\ref{as:B}, \ref{as:f1}--\ref{as:f5} hold together with one of the following additional assumptions: 
\begin{enumerate}
    \item[(A)] $\displaystyle \lim_{|(n,t)| \to \infty} f(n,t,s)=0$ locally uniformly in $s\in\R$.\label{as:f:locnt}
 \item[(B)] There exist $T>0$ and $N \in \N$ such that $f(n,t,s)=f(n,t+T,s)$, $f(n,t,s)=f(n+N,t,s)$ and $d(t)=d(t+T)$ for all $(n,t,s) \in \Z \times \R \times \R$ and $M, B$ commute with $S_N$. \label{as:f:pernt}
 \item[(C1)] There exists $T>0$ such that $f(n,t,s)=f(n,t+T,s)$,  $d(t)=d(t+T)$ for all $(n,t,s) \in \Z \times \R \times \R$ and $\lim_{|n| \to \infty} f(n,t,s)=0$ locally uniformly in $s\in\R$ and uniformly in $t \in [0,T]$.     \label{as:f:locn:pert}
\item[(C2)] There exists $N \in \N$ such that $f(n,t,s)=f(n+N,t,s)$ for all $(n,t,s) \in \Z \times \R \times \R$, the operators $M, B$ commute with $S_N$, and $\lim_{|t| \to \infty} f(n,t,s)=0$ locally uniformly in $s\in\R$ and uniformly in $n \in \{0,\ldots, N-1\}$ .  \label{as:f:loct:pern}
\end{enumerate}
Then there exists a least-energy rogue wave solution $u\in H^2(\R;\ell^2(\Z))$ to \eqref{eq:main}. In particular $u\in C_0^1(\R;\ell^2(\Z))$. The same holds with $f$ replaced by $-f$.
\end{theorem}
\begin{remark} 
Condition \ref{as:f:locnt} means that $f$ is localized in $t$ and $n$ whereas condition \ref{as:f:pernt} asserts a periodic structure of \eqref{eq:main} in both space and time. Assuming either condition \ref{as:f:locn:pert} or \ref{as:f:loct:pern}, we require a periodic structure in only one of the variables time or space, and that $f$ is localized with respect to the complementary space or time variable. The terminology \emph{least-energy solution} is explained at the beginning of Section~\ref{sec:proofmain}.
\end{remark}

\begin{corollary} \label{cor:breathers}
    Under the same assumptions  \ref{as:M}--\ref{as:B}, \ref{as:f1}--\ref{as:f5} as in Theorem~\ref{thm:main} together with one of the $T$-periodic cases \ref{as:f:pernt} or \ref{as:f:locn:pert} there exists for every $k\in \N$ a $kT$-periodic least-energy breather solution $u\in H^2_\mathrm{per}(\R;\ell^2(\Z))\cap C_\mathrm{per}^1(\R;\ell^2(\Z))$ to \eqref{eq:main}. The same holds with $f$ replaced by $-f$.
\end{corollary}

\begin{remark}
    It is reasonable to conjecture that a subsequence of the least-energy $kT$-periodic breather solutions converges to a least-energy rogue wave solution of Theorem~\ref{thm:main} as $k\to \infty$. This is an open problem which, due to its complexity, we leave for future work. In the case of the nonlinear Schr\"odinger equation an argument of this type has been carried out in \cite{pankov}. 
\end{remark}

Let us bring our main result from Theorem~\ref{thm:main} on rogues waves/homoclinic solutions into perspective with previous results from the literature. 

\underline{Homoclinics in the scalar ODE case:} As an application of simplest form take $f(n,t,s)=-s^3$, $B=M=\Id$ and $u(t)=(\ldots,0,\ldots, 0, u_0(t), 0, \ldots,0,\ldots)$ a sequence where only the $n=0$-mode is present. In this case \eqref{eq:main} becomes the ODE $\frac{1}{d(t)} \ddot u_0 + u_0+  u_0^3=0$ for which we are seeking a homoclinic orbit. Obviously, for constant $d>0$ only periodic solutions and no homoclinics exist. From the linear ODE perspective a homoclinic requires an exponential dichotomy, i.e., a pair of linearly independent solutions of $\frac{1}{d(t)} \ddot u + u=0$ tending to $0$ as $t\to\pm\infty$, respectively. This, in turn, is in one-to-one correspondence with $1\not \in \sigma(-\frac{1}{d(t)} \frac{\mathrm{d}^2}{\mathrm{d}t^2})$, cf. \cite{sell}, which is identical with our assumption \ref{as:specgap}. Turning the necessary linear condition into a sufficient condition for the nonlinear ODE to possess a homoclinic orbit, however, requires specific nonlinearities, e.g., $f(n,t,s)=-s^3$. From a variational perspective, the homoclinic orbit can be obtained as a ground state of the functional $J(u)= \int_\R \dot u^2 - d(t) u^2 - \frac{1}{2} d(t) u^4\der t$ on the space $H^1(\R)$ and the existence of such a ground state is well known, see, e.g., \cite{jeanjean} and \cite{pankov}, provided $1\not \in \sigma(-\frac{1}{d(t)} \frac{\mathrm{d}^2}{\mathrm{d}t^2})$ and $d$ is periodic. 

\underline{Homoclinics in the finite dimensional ODE case:}
Next we  consider homoclinic solutions of \eqref{eq:main} in the finite-dimensional case with state space $\ell^2(\{1,\dots, N\})$ considered via extension by zero as a subspace of $\ell^2(\Z)$. In this case, \eqref{eq:main} reads as $\frac{1}{d(t)}\ddot u + Mu = \nabla_u I(t,u)$ where $M\in \R^{N\times N}_\text{sym}$ is positive definite and $I(t,u) \coloneqq \sum_{n=1}^N F(n,t, (Bu)_n(t))$. For such systems, but with $M$ negative definite,  classical results of Coti--Zelati, Rabinowitz \cite{coti_zelati_rabinowitz} apply. They used a variational mountain pass approach where the quadratic part of the functional was definite. The case where $M\in \R^{N\times N}_\text{sym}$ is positive definite and $\sigma(M)\cap \sigma(-\frac{1}{d(t)}\frac{\mathrm{d}^2}{\mathrm{d}t^2})=\emptyset$ (like in our case), which forces the quadratic part of the underlying functional to be indefinite, has been investigated in \cite{arioli_szulkin}. In both papers, the structural assumptions on $I$ are more general than the case considered in the present work. On the other hand, as can be seen already for $N=1$ due to examples in \cite[Chapter 3.2]{szulkin_weth}, our assumptions on $f, F$ cover cases which are not covered by \cite{coti_zelati_rabinowitz, arioli_szulkin}.

\underline{Homoclinics/rogues waves in $\ell^2(\Z)$ for nonlinear wave equations:} The overwhelming majority of results on discrete rogue waves seem to have been produced for discrete NLS models, cf. \cite{flach_gorbach} for a survey, which are very different from the discrete nonlinear wave equation. This is not surprising since it is the richness of patterns in continuous NLS that allows to study the persistence of, e.g., rogue waves, under perturbations like discretization. For \eqref{eq:main_0} we are not aware of any result on the existence of a rogue wave. Instead, existence results for rogue waves in spatially continuous versions of \eqref{eq:main} with $B=\Id$ are given in \cite{henninger}. Discrete breathers in nonlinear Klein--Gordon lattices like \eqref{eq:main_0} have been constructed, e.g., in \cite{bambusi_et_al} and \cite{Hennig_Karachalios}.

\underline{Homoclinics/rogues waves in $\ell^2(\Z)$ for the FPUT model:} One approach to rouge waves in an FPUT lattice is to use a multiple scale ansatz and derive a modulation equation such as NLS or mKDV which have an exact rogue wave solution, cf. \cite{Chong_Lee_etal, chong_kevrekidis_etal,Kevrekisid_yang_sun_liu}. The idea is then that the rogue wave pattern arising from the modulation equation can also be (numerically) observed in the original FPUT lattice. A similar approach to 2d FPUT lattices or FPUT systems can be found in \cite{miyazawa_chong_kevrekidis_yang, yang_sun}. While the multiple scale ansatz combined with numerical methods provides very detailed information on the rogue waves, its rigorous justification is difficult and it is often restricted to a perturbative regime. In contrast, our purely analytical rigorous approach provides less details on the actual shape of the rogue wave but has the advantage of being non-perturbative. Additionally, we need two main features which have not been present in the multiple scale approaches: one is the time-varying coefficient $1/d(t)$ in front of the kinetic term $\frac{\mathrm{d}^2}{\mathrm{d}t^2}$ and the other is the presence of the added positive local forces $q_n u_n$, $\inf_\Z q>0$ in \eqref{eq:main_00}, cf. Corollary~\ref{cor:zitierbar}. Without either of them the essential condition \ref{as:specgap} cannot be fulfilled. Finally, let us mention three sources for breathers in FPUT-chains: using spatial dynamics and center manifold reduction James and Noble~\cite{james_noble} have shown existence of FPUT breathers with positive linear restoring forces and squared frequencies $k^2\omega^2$ in spectral gaps of the spatial linear operator, which in our case amounts to the condition that $\sigma(-\frac{1}{d(t)}\frac{\mathrm{d}^2}{\mathrm{d}t^2})_{L^2\text{per}([0,kT])}\cap \sigma(M)=\emptyset$. When the linear restoring forces are negative (which corresponds to $+\Delta_1$ instead of $-\Delta_1$ in \eqref{eq:main_00}), Arioli and Gazzola~\cite{arioli_gazzola} used variational methods to prove existence of breathers. A systematic variational approach to FPUT-breathers is given in Pankov~\cite{PankovFPU}. In all three approaches $d=\const$ and there were no local forces, i.e., $q=0$.

\subsection*{Structure of the paper and notation}
For $k\in \N_0$ and $q,q_1,q_2\in [1,\infty]$ we use the standard Lebesgue and Sobolev spaces $L^q(\R)$, $H^k(\R)$ as well as the Bochner spaces $L^{q_1}(\R;\ell^{q_2}(\Z))$, $H^k(\R;\ell^2(\Z))$, $C_0(\R;\ell^q(\Z))$, $C_0^1(\R;\ell^q(\Z))$. When $d:\R \to [0,\infty)$ is a measurable, non-negative function we use the weighted space $L_d^2(\R)$ equipped with the norm ${\norm{u}_{L_d^2(\R)}^2}=\int_\R |u(t)|^2 d(t) \der t$ and extend this definition to the weighted Bochner space $L_d^2(\R;\ell^2(\Z))$ with the norm ${\norm{u}_{L_d^2(\R;\ell^2(\Z))}^2}=\int_\R \sum_{n\in \Z} |u_n(t)|^2 d(t) \der t$. For an element $u$ in the Bochner space $L^2_d(\R;\ell^2(\Z))$ we write $u(t) = (u(t)_n)_{n\in \Z}\in \ell^2(\Z)$ for all $t\in \R$ and, for convenience, we use the notation $u(t)_n = u_n(t)$. The same notation is applied for elements of the other Bochner spaces as well. For a bounded, linear operator $A:X \to Y$, where $X$ and $Y$ are Banach spaces, we denote its operator norm by $\norm{A}_{\mathrm{op}, X \to Y}$ with the convention $\norm{A}_{\mathrm{op}, X \to X}=\norm{A}_{\mathrm{op}, X}$ whenever $A$ is a bounded endomorphism. We use the notation $\mathcal{Z} \subset \Z$ and $\mathcal{R} \subset \R$ to denote subsets of $\Z$ and $\R$, respectively.

From now on we assume that \ref{as:M}--\ref{as:B}, \ref{as:f1}--\ref{as:f5} are satisfied.
The rest of the paper is structured as follows. In Section~\ref{sec:varset} we construct a functional $J$ on the Hilbert space $\mathcal{H}=H^1(\R;\ell^2(\Z))$ such that critical points of $J$ are weak solutions of \eqref{eq:main}. This requires some properties of the wave operator $L\coloneqq \frac{1}{d(t)}\partial_t^2+M$ and its absolute value $|L|$. Furthermore, we give embedding properties of $\mathcal{H}$ and provide versions of P.L.~Lions' concentration compactness argument suitable for our space-time setting. In Section~\ref{sec:proofmain} we use a saddle-point technique to show the existence of a critical point of $J$ in $\mathcal{H}$ and analyze its further regularity properties. This will then complete the proof of our main result of Theorem~\ref{thm:main}. We finish this section with the proof of the existence of breathers stated in Corollary~\ref{cor:breathers}.
Finally, equations \eqref{eq:main_0} and \eqref{eq:main_00} are covered in Appendix~\ref{sec:examples} where we show that all conditions needed to apply Theorem~\ref{thm:main} can be satisfied for $M=-\Delta_1+q_n$ by a suitable choice of $(q_n)_{n\in\Z}$ and a specific class of periodic two-step potentials $d$.   
    \section{Variational setting}\label{sec:varset}

So far we considered the operator $\frac{1}{d(t)}\frac{\der^2}{\der t^2}$ as a self-adjoint operator on $H^2(\R)\subset L^2_d(\R)\to L^2_d(\R)$. This operator can be extended to the Bochner space $H^2(\R;\ell^2(\Z))\subset L^2_d(\R;\ell^2(\Z))$ by setting
$$
\left(\frac{1}{d(t)}\frac{\der^2}{\der t^2} u\right)(t):= 
 \left(\frac{1}{d(t)}\frac{\der^2}{\der t^2} u_n(t)\right)_{n\in \Z}
$$
since $t\mapsto u_n(t)$ lies in $H^2(\R)$ for every $n\in\Z$. In this way $\frac{1}{d(t)}\frac{\der^2}{\der t^2}: H^2(\R; \ell^2(\Z))\subset L^2_d(\R;\ell^2(\Z))\to L^2_d(\R;\ell^2(\Z))$ is a self-adjoint operator and its resolvent set (and hence its spectrum) stays the same.
Similarly, in assumption~\ref{as:M} the operator $M$ was considered as a continuous endomorphism on $\ell^2(\Z)$. It trivially extends to an endomorphism of $L^2_d(\R; \ell^2(\Z))$ by setting 
$$
(Mu)(t):= \bigl((Mu(t))_n\bigr)_{n\in\Z} 
$$
since $u(t)\in \ell^2(\Z)$ for almost all $t\in \R$. The extended operator $M: L^2_d(\R; \ell^2(\Z))\to L^2_d(\R; \ell^2(\Z))$ is bounded and self-adjoint and has the same resolvent set and hence the same spectrum as $M: \ell^2(\Z)\to \ell^2(\Z)$. Likewise, we can extend the shift operator $S_N$ as well as $B, B^\ast$ from \ref{as:B} to endomorphisms of $L^2_d(\R; \ell^2(\Z))$ and see that the extension of $B^\ast$ is still the adjoint of the extension of $B$. For $M, B$ and $u\in L^2_d(\R;\ell^2(\Z))$ we use for all $t\in \R$ and all $n\in \Z$ the pointwise notation
$$
(Mu)_n(t) \coloneqq \bigl(Mu(t)\bigr)_n, \quad (Bu)_n(t) \coloneqq \bigl(Bu(t)\bigr)_n.
$$
As the following two lemmas show, $B$ and $M$ extend as bounded endomorphisms also to other spaces. 

\begin{lemma}\label{lem:Bbounded}
If $B$ satisfies \ref{as:B} then 
\begin{enumerate} 
    \item $B: H^1(\R;\ell^2(\Z)) \to H^1(\R;\ell^2(\Z))$ is bounded and injective,
    \label{lem:Bbounded:a}
    \item
    $B:L^q(\R; \ell^q(\Z)) \to L^q(\R; \ell^q(\Z))$ is bounded for all $q \in [2,\infty]$. \label{lem:Bbounded:b} 
 \end{enumerate}
\end{lemma}

\begin{proof}
(a) For boundedness, we take $u \in H^1(\R;\ell^2(\Z))$ and calculate
    \begin{align*}
        \norm{Bu}_{H^1(\R;\ell^2(\Z))}^2
        &= \int_\R \norm{(B\dot{u})(t)}_{\ell^2(\Z)}^2 \der t+ \int_\R \norm{(Bu)(t)}_{\ell^2(\Z)}^2 \der t   \\
        &\leq  \norm{B}_{\mathrm{op},\ell^2}^2\left(\int_\R \norm{\dot{u}(t)}_{\ell^2(\Z)}^2 \der t+ \int_\R \norm{u(t)}_{\ell^2(\Z)}^2 \der t \right)  \\
        &= \norm{B}_{\mathrm{op},\ell^2}^2 \norm{u}_{H^1(\R;\ell^2(\Z))}^2.
    \end{align*}
For injectivity, suppose $0=\norm{Bu}_{H^1(\R;\ell^2(\Z))}^2 \geq \int_\R \norm{(Bu)(t)}_{\ell^2(\Z)}^2 \der t$. This forces $\norm{(Bu)(t)}_{\ell^2(\Z)}=0$ for almost all $t\in \R$ so that injectivity of $B: \ell^2(\Z)\to \ell^2(\Z)$ implies $u(t)=0$ for almost all $t\in \R$.

(b) First, since $B$ is bounded on $\ell^2(\Z)$ and on $\ell^\infty(\Z)$, we obtain by Riesz--Thorin interpolation theorem, see \cite{lunardi}, that $B:\ell^q(\Z) \to \ell^q(\Z)$ is bounded also for all $q \in (2,\infty)$. As in \ref{lem:Bbounded:a} we estimate for all $u \in L^q(\R;\ell^q(\Z))$ 
     \begin{align*}
        \norm{Bu}_{L^q(\R;\ell^q(\Z))}^q= \int_\R \norm{(Bu)(t)}_{\ell^q(\Z)}^q \der t\leq \norm{B}_{\mathrm{op},\ell^q}^q \int_\R \norm{u(t)}_{\ell^q(\Z)}^q  \der t =\norm{B}_{\mathrm{op},\ell^q}^q \norm{u}_{L^q(\R;\ell^q(\Z))}^q. &\qedhere
    \end{align*}  
\end{proof}

A similar statement holds for extensions of the operator $M: \ell^2(\Z)\to \ell^2(\Z)$ as the next lemma shows. It also shows a commutator property for $\sqrt{M}$ which will help us to see that if \eqref{eq:main} is invariant under translations in space, then our variational problem has the same invariance.
\begin{lemma}\label{lem:Mbounded} \label{lem:operatorcommute}
If $M$ is as in assumption \ref{as:M} then the following holds:
\begin{itemize}
    \item[(i)] $M: H^k(\R; \ell^2(\Z)) \to H^k(\R;\ell^2(\Z))$ is bounded for $k=0,1,2$.
    \item[(ii)] If $MS_N=S_NM$ as in assumption~\ref{as:f:pernt} or \ref{as:f:loct:pern} then $ \sqrt{M}S_N=S_N \sqrt{M}$.
\end{itemize}
\end{lemma}

\begin{proof}
We omit the proof of (i) since it is based on the same arguments as the proof of Lemma~\ref{lem:Bbounded}. The proof of (ii) follows from \cite[Theorem 4.1.3]{kr}.
\end{proof}

With these observations at hand we introduce the wave operator $L$ as
\begin{align*}
    L \coloneqq \frac{1}{d(t)}\frac{\der^2}{\der t^2} +M: H^2(\R;\ell^2(\Z)) \subset L^2_d(\R; \ell^2(\Z)) \to L^2_d(\R; \ell^2(\Z)).
\end{align*}
We note that $L$ is bounded from above and self-adjoint and $\sigma(L)= \sigma(\frac{1}{d(t)}\frac{\der^2}{\der t^2})+\sigma(M)$, cf. \cite[Corollary to Theorem VIII.33]{ReedSimon_Vol1}\footnote{Here we use that $L^2_d(\R;\ell^2(\Z))$ is isomorphic to the Hilbert tensor product $L^2_d(\R)\otimes\ell^2(\Z)$ and likewise $H^k(\R; \ell^2(\Z))$ is isomorphic to the Hilbert tensor product $H^k(\R)\otimes\ell^2(\Z)$ for $k=1,2$. Since the operators $\frac{1}{d(t)}\frac{\der^2}{\der t^2}$ on $H^2(\R)$ and $M$ on $\ell^2(\Z)$ extend to the Hilbert tensor product $H^2(\R)\otimes\ell^2(\Z)$ as $\frac{1}{d(t)}\frac{\der^2}{\der t^2}\otimes \Id$ and $\Id\otimes M$ we get that $\sigma(L)= \overline{\sigma(\frac{1}{d(t)}\frac{\der^2}{\der t^2})+\sigma(M)}= \sigma(\frac{1}{d(t)}\frac{\der^2}{\der t^2})+\sigma(M)$ since $\sigma(M)$ is compact.}. Moreover, assumption \ref{as:specgap} implies $0\not \in \sigma(L)$.
The quadratic form $b_L:\mathcal{H} \times \mathcal{H} \to \R$ of $L$ is given by
\begin{align*}
    b_L(u,\varphi)=\int_\R \sum_{n \in \Z}\left( - \frac{1}{d(t)} \dot{u}_n(t) \dot{\varphi}_n(t) + (\sqrt{M}u)_n(t) (\sqrt{M}\varphi)_n(t)  \right) d(t) \der t
\end{align*}
on the Hilbert space 
\begin{align*}
\mathcal{H}\coloneqq H^1(\R;\ell^2(\Z))= \{ u: \R \to \ell^2(\Z) \text{ s.t. } \norm{u} < \infty \}
\end{align*}
equipped with its canonical norm
\begin{align}\label{eq:norm}
    \|u\|^2\coloneqq \int_\R \sum_{n \in \Z}  \bigl(|\dot u_n(t)|^2+ |u_n(t)|^2\bigr) \der t.
\end{align}
We also consider the self-adjoint operator $|L|: H^2(\R;\ell^2(\Z)) \subset L^2_d(\R; \ell^2(\Z)) \to L^2_d(\R; \ell^2(\Z))$ and its quadratic form $b_{|L|}:\mathcal{H} \times \mathcal{H} \to \R$. By Kato's second representation theorem, cf. \cite[Theorem 2.23, Chapter $6$, §2]{kato}, we know that $D(\sqrt{|L|})=D(b_{|L|})=\mathcal{H}$ and 
\begin{align*}
    b_{|L|}(u,v)=\langle \sqrt{|L|} u,  \sqrt{|L|} v \rangle_{L^2_d(\R; \ell^2(\Z))} \quad \mbox{ for all } u,v \in \mathcal{H}.
\end{align*}
Since the operator $|L|$ is strictly positive by assumption \ref{as:specgap}, we define a norm on $\mathcal{H}$ via
\begin{align*}
    \|u\|_{\mathcal{H}}^2\coloneqq b_{|L|}(u,u)= \norm{\sqrt{|L|}u}_{L^2_d(\R; \ell^2(\Z))}^2  \quad \text{ for } u \in \mathcal{H}.
\end{align*}
This norm is equivalent to the canonical norm $\norm{\cdot}$ on $\mathcal{H}$ introduced in \eqref{eq:norm} as the following lemma shows.

\begin{lemma}
    There exist constants $c_1,c_2>0$ such that
    \begin{align}\label{eq:equinorm}
     c_1   \norm{u} \leq  \norm{u}_{\mathcal{H}} \leq c_2 \norm{u}
    \end{align}
    for all $u \in \mathcal{H}$.
\end{lemma}
\begin{proof}
For readability we use the notation $L=L_1+L_2$ where
\begin{align*}
 L_1:& H^2(\R;\ell^2(\Z)) \subset L^2_d(\R; \ell^2(\Z)) \to L^2_d(\R;\ell^2)(\Z)), \hspace*{-4em}  &&u \mapsto L_1u=\frac{1}{d(t)} \frac{\der^2}{\der t^2} u , \\
L_2:& L^2_d(\R; \ell^2(\Z)) \to L^2_d(\R;\ell^2(\Z)), &&u \mapsto L_2u=M u
\end{align*}
and we have already seen that $L_1$ and $L_2$ are self-adjoint with $L_2$ being bounded.
Furthermore, since the domains of $L_2L_1$ and $L_1L_2$ fulfill
\begin{align*}
    H^2(\R;\ell^2(\Z))= D(L_2L_1)=D(L_1 L_2)
\end{align*}
and since $L_1, L_2$ commute we can utilize \cite[Theorem 3.1]{boucif} to get
\begin{align*}
    |L_1+L_2| \leq |L_1| + |L_2|= -L_1+L_2
\end{align*}
in the sense of the Loewner ordering. Hence we can estimate
\begin{align*}
\norm{u}_\mathcal{H}^2=\norm{\sqrt{|L|} u}_{L^2_d(\R; \ell^2(\Z))}^2 &\leq \norm{\sqrt{-L_1}u}_{L^2_d(\R; \ell^2(\Z))}^2 + \norm{\sqrt{L_2}u}_{L^2_d(\R; \ell^2(\Z))}^2  \\
&\leq c \left( \norm{\dot u}_{L^2(\R; \ell^2(\Z))}^2+ \norm{ u}_{L^2(\R; \ell^2(\Z))}^2  \right) \\
&=c \norm{u}^2
\end{align*}
for a constant $c>0$ and all $u \in \mathcal{H}$. This shows one part of inequality \eqref{eq:equinorm} with $c_2=\sqrt{c}$.

\medskip

Next we claim that there exists a constant $c_1>0$ such that
\begin{align*}
   \norm{u}_{\mathcal{H}}^2=\norm{\sqrt{|L|} u}_{L^2_d(\R;\ell^2(\Z))}^2 \geq c_1^2\left(\norm{\dot u}_{L^2(\R; \ell^2(\Z))}^2 + \norm{u}_{L^2(\R;\ell^2(\Z))}^2 \right)
\end{align*}
for all $u  \in \mathcal{H}$. Suppose the claim is wrong, i.e., for every $j \in \N$ there exists $u^j \in \mathcal{H}$ such that
\begin{align}\label{eq:loewner}
 \norm{u^j}_\mathcal{H}^2=  \norm{\sqrt{|L|} u^j}_{L^2_d(\R\; \ell^2(\Z))}\to 0 \quad \mbox{ and }  \quad \norm{\dot u^j}_{L^2(\R; \ell^2(\Z))}^2 + \norm{u^j}_{L^2(\R;\ell^2(\Z))}^2=1.
\end{align}
Since $\norm{\sqrt{|L|}u^j}_{L^2_d(\R\; \ell^2(\Z))}^2 \geq \inf\sigma(|L|)\norm{u^{j}}_{L^2(\R; \ell^2(\Z))}^2$ and $\inf\sigma(|L|)>0$ by assumption \ref{as:specgap} we find $\norm{u^j}_{L^2(\R;\ell^2(\Z))} \to 0$. Together with \eqref{eq:loewner} this implies $\norm{\dot u^j}_{L^2(\R; \ell^2(\Z))}\to 1$ as $j\to \infty$ and (by boundedness of $L_2=M$) also $\norm{\sqrt{L_2}u^j}_{L^2_d(\R;\ell^2(\Z))} \to 0$ as $j \to \infty$. As before, in the sense of the Loewner ordering we have $|L_1+L_2| \geq |L_1| - |L_2|= -L_1-L_2$ so that we get the contradiction
\begin{align*}
   0 \stackrel{j\to \infty}{\longleftarrow} \norm{u^j}_{\mathcal{H}}^2 = \norm{\sqrt{|L_1 +L_2|}u^j}_{L^2_d(\Z \times \R)}^2 
    & \geq \norm{\sqrt{-L_1}u^j}_{L^2_d(\Z \times \R)}^2 - \norm{\sqrt{L_2}u^j}_{L^2_d(\Z \times \R)}^2\\
    & = \norm{\dot u^j}_{L^2(\Z \times \R)}^2 - \norm{\sqrt{L_2}u^j}_{L^2_d(\Z \times \R)}^2\stackrel{j\to \infty}{\longrightarrow} 1.
    \qedhere
\end{align*}
\end{proof}

For the variational approach used in this paper we consider the functionals
\begin{align*}
    J_0(u)&\coloneqq \frac12 b_L(u,u)=\int_\R \sum_{n \in \Z} \left(-\frac{1}{2d(t)} | \dot{u}_n(t)|^2 + \frac{1}{2} |(\sqrt{M} u)_n(t)|^2 \right)d(t) \der t, \\
    J_1(u)&\coloneqq \int_\R\sum_{n \in \Z} F(n,t,(Bu)_n(t)) d(t)\der t, \\
    J(u) & \coloneqq J_0(u)-J_1(u)
\end{align*}   
for $u \in \mathcal{H}$. While $J_0$ is well-defined on $\mathcal{H}$ we show next (see Theorem~\ref{theom:emb} below) that also $J_1$ and hence $J$ are well-defined on $\mathcal{H}$. It is then standard to show the continuous Fr\'{e}chet-differentiability of $J$ and thus critical points of $J$ are weak solutions of \eqref{eq:main} in the following sense.
\begin{definition}\label{def:weak_sol}
     A function $u\in \mathcal{H}$ is called a weak solution to \eqref{eq:main} if
     $$    
     \int_\R \sum_{n \in \Z} \left(- \dot u_n(t) \dot\varphi_n(t) + (\sqrt{M}u)_n(t) (\sqrt{M}\varphi)_n(t) d(t)\right)\der t = \int_\R \sum_{n\in\Z} f(n,t,(Bu)_n(t)) (B\varphi)_n(t) d(t) \der t
     $$
holds for all $\varphi\in \mathcal{H}$.
 \end{definition}
 As we shall see later, once we have a weak solution $u\in \mathcal{H}$ it is not difficult to show $u\in H^2(\R;\ell^2(\Z))$.

\begin{theorem}\label{theom:emb} Let $q_1,q_2\in [2,\infty]$. The embedding $\mathcal{H}\hookrightarrow C_0(\R;\ell^{q_2}(\Z))$ is continuous. If $\mathcal{R}\subset \R$ is an interval and $\mathcal{Z}\subset \Z$ then the embeddings $\mathcal{H}\hookrightarrow L^{q_1}(\mathcal{R}; \ell^{q_2}(\mathcal{Z}))$, $\mathcal{H}\hookrightarrow C(\mathcal{R};\ell^{q_2}(\mathcal{Z}))$ are continuous. They are compact whenever $\mathcal{R}$ and $\mathcal{Z}$ are additionally bounded.
\end{theorem}

\begin{proof}
The continuous embedding $H^1(\R) \hookrightarrow L^{q_1}(\mathcal{R})$ translates by summation over $n\in \mathcal{Z}$ and $\ell^{2}(\mathcal{Z})\hookrightarrow \ell^{q_2}(\mathcal{Z})$ to the embedding $\mathcal{H}\hookrightarrow L^{q_1}(\mathcal{R}; \ell^{q_2}(\mathcal{Z}))$. Moreover, for $t,s\in \mathcal{R}$ and $n\in \mathcal{Z}$ the estimate
$$
|u_n(t)-u_n(s)| \leq \sqrt{|t-s|} \|\dot u_n\|_{L^2(\mathcal{R})} 
$$
combined with summation over $n\in \mathcal{Z}$ and $\ell^{2}(\mathcal{Z})\hookrightarrow \ell^{q_2}(\mathcal{Z})$ yields that $u\in C(\mathcal{R}; \ell^{q_2}(\mathcal{Z}))$. Compactness follows from the Marcel Riesz/Arz\'{e}la--Ascoli criterion whenever $\mathcal{R}\subset \R$ and $\mathcal{Z}\subset \Z$ are bounded. To see that $u\in C_0(\R;\ell^{q_2}(\Z))$ note that 
$$
\sum_{n\in \mathcal{Z}} u_n(t)^2 = \sum_{n\in \mathcal{Z}} u_n(0)^2+2\int_0^t \sum_{n\in \mathcal{Z}} u_n(\tau) \dot u_n(\tau)\der \tau.
$$
Since the integral converges this identity implies that $\lim_{|t|\to \infty} \|u(t)\|_{\ell^2}$ exists and equals $0$. Finally, by the $\ell^{q_2}(\Z) \hookrightarrow \ell^2(\Z)$ embedding we obtain $u\in C_0(\R;\ell^{q_2}(\Z))$.
\end{proof}

We continue with collecting some properties of the functions $f$ and $F$ based on our assumptions.
\begin{lemma}\label{lem:est_on_f} The following properties of the functions $f$ and $F$ hold.
    \begin{enumerate}[(i)]
    \item $\frac{1}{2}f(n,t,s)s > F(n,t,s)>0$ for all $s \neq 0$ and $(n,t)\in \Z \times \R$.\label{lem:est_on_f:fF}
        \item For every $\varepsilon >0$ there exists $C_\varepsilon >0$ such that $|f(n,t,s)|\leq \varepsilon |s| + C_\varepsilon |s|^p$ and $|F(n,t,s)| \leq \varepsilon |s|^2+ C_\varepsilon |s|^{p+1}$ for all $(n,t,s) \in \Z \times \R \times \R$.\label{lem:est_on_f:gen_est}
        \item Let $\rho>0$ and $r\geq 1$ be fixed. For every $\varepsilon >0$ there exists a bounded function $h_\varepsilon:\Z \times \R \to [0,\infty)$ such that
        \begin{align*}
            |f(n,t,s)|^r \leq \varepsilon |s|^r + h_\varepsilon(n,t)|s|^{r(p+\rho)} \quad \mbox{ for all } (n,t,s) \in \Z \times \R \times \R
        \end{align*} 
        and
        \begin{enumerate}
            \item $\lim_{|(n,t)| \to \infty} h_\varepsilon(n,t)=0$ if $f$ satisfies assumption~\ref{as:f:locnt}, 
            \item $\lim_{|n| \to \infty} h_\varepsilon(n,t)=0$ uniformly in $t \in \R$ and $h_\varepsilon (n,t)=h_\varepsilon(n,t+T)$ for all $n \in \Z$, $t \in \R$ if $f$ satisfies assumption~\ref{as:f:locn:pert}. 
            \item $\lim_{|t| \to \infty} h_\varepsilon(n,t)=0$ uniformly in $n \in \N$ and $h_\varepsilon (n,t)=h_\varepsilon(n+N,t)$ for all $n \in \Z$, $t \in \R$ if $f$ satisfies assumption~\ref{as:f:loct:pern}.
        \end{enumerate}    \label{lem:est_on_f:eps}
        \item If $\mathcal{Z}\subset\Z$ and $\mathcal{R}\subset \R$ is an interval then the Nemytskii operator $\tilde{f}: \mathcal H \to L^2(\mathcal{R}; \ell^2(\mathcal{Z}))$ defined by $\tilde f(u)(n,t)\coloneqq f(n,t,u_n(t))$ for all $n\in \mathcal{Z}$ and $t\in \mathcal{R}$ is continuous and it is compact whenever $\mathcal{R}$ and $\mathcal{Z}$ are additionally bounded.
      \end{enumerate}
\end{lemma}

\begin{remark}
    With the notation used in \eqref{eq:main} we have for $u\in \mathcal{H}$ the relation that $f^\sharp(t, u(t))= \tilde f(u)(t)$ for almost all $t\in \mathcal{R}$ and $\tilde f(u) = \{ \mathcal{R} \ni t \mapsto f^\sharp(t,u(t))\}$. From a pointwise perspective, one finds $(f^\sharp(t,u))_n=\tilde f (u)(n,t)$ for almost all $t \in \mathcal{R}$ and all $n\in \mathcal{Z}$.
\end{remark}

\begin{proof}
(i) We first consider $s>0$ where $f(n,t,s)>0$ by assumption~\ref{as:f2} and \ref{as:f3}. Since $s \mapsto \frac{f(n,t,s)}{s}$ is strictly increasing by assumption~\ref{as:f3} we have
\begin{align*}
    0<F(n,t,s)= \int_0^1 f(n,t,\xi s)s \der \xi < \int_0^1 f(n,t,s) \xi s \der \xi = \frac{1}{2} f(n,t,s)  s
\end{align*}
for all $(n,t,s) \in \Z \times \R \times (0,\infty)$. Similarly one obtains the estimate for $s<0$. \\
(ii) This follows from assumptions \ref{as:f2} and \ref{as:f5}. \\
(iii) Let $\rho, \varepsilon >0$. It is enough to prove the statement for $r=1$ from which the general case follows by taking the $r$-th power. By assumption~\ref{as:f2} there exists $R=R(\varepsilon)>0$ such that $|f(n,t,s)| \leq \varepsilon |s| $ for all $|s| \leq R$ and $(n,t) \in \Z \times \R$. Using \ref{as:f5} and $\rho >0$ we define $h_{\varepsilon}:\Z \times \R \to [0,\infty)$ as
        \begin{align*}
        h_{\varepsilon}(n,t)\coloneqq\sup_{|s| \geq R} \frac{|f(n,t,s)|}{|s|^{p+\rho}}
        \end{align*}
        which altogether yields $|f(n,t,s)| \leq \varepsilon |s| + h_{\varepsilon}(n,t)|s|^{p+\rho}$ for all $(n,t,s)\in \Z \times \R \times \R$. Moreover, by \ref{as:f5} and since $\rho>0$ there exists $\tilde R= \tilde R(\varepsilon)$ such that
        \begin{align} \label{eq:est_f}
        |f(n,t,s)| \leq C(1+|s|^p)\leq \varepsilon |s|^{p+\rho} \mbox{ for all } |s|\geq \tilde R, (n,t)\in \Z\times \R
        \end{align}
        and w.l.o.g. $\tilde R\geq R$. In particular, $h_\varepsilon$ is bounded.
        In case \ref{as:f:locnt} we have \eqref{eq:est_f} also for $|(n,t)|\geq K(\varepsilon)$ and $R\leq |s|\leq \tilde R$ which shows that $\lim_{|(n,t)| \to \infty} h_{\varepsilon}(n,t)=0$. 
        In case \ref{as:f:locn:pert} we get \eqref{eq:est_f} for $|n|\geq K(\varepsilon)$, $t\in [0,T]$ and $R\leq |s|\leq \tilde R$. Since $f$ and hence $h_\varepsilon$ are periodic in $t$ it follows that $\lim_{|n| \to \infty} h_{\varepsilon}(n,t)=0$ uniformly in $t \in \R$. In case \ref{as:f:loct:pern} the inequality \eqref{eq:est_f} also holds true for $|t| \geq K(\varepsilon)$, $n\in \{0,\ldots, N-1\}$ and $R\leq |s| \leq \tilde R$. By the periodicity of $f$ with respect to $n$ we see $\lim_{|t| \to \infty} h_{\varepsilon}(n,t)=0$ uniformly in $n \in \N$.\\
(iv) This is a consequence of (ii), the continuous/compact embedding results $\mathcal{H}\hookrightarrow L^2(\mathcal{R}; \ell^2(\mathcal{Z}))\cap L^{2p}(\mathcal{R}; \ell^{2p}(\mathcal{Z}))$ from Theorem~\ref{theom:emb} and a dominated convergence argument, see \cite[Chapter 1.2 and Theorem 2.2]{ambrosetti_prodi} for details.
\end{proof}
In view of variational arguments for the functional $J$ the lack of global compactness of the embeddings in Theorem~\ref{theom:emb} needs to bypassed. In case \ref{as:f:locnt} of Theorem~\ref{thm:main} this is done by assuming $f$ to be localized in $(n,t)$. In the case \ref{as:f:pernt} we do this by assuming that \eqref{eq:main} has a periodic structure both in space and time which allows us to use P.L.~Lions' concentration compactness argument, cf. \cite[Lemma~I.1]{Lions_org}. We will prove a version of Lions' argument as given in Theorem~\ref{lem:ccot}. In case \ref{as:f:locn:pert} the equation \eqref{eq:main} has a periodic structure only in time. In this case we give another version of Lions' argument in Theorem~\ref{thm:cc}~\ref{lem:cct}, and later combine it with the localization property of $f$ with respect to the spatial variable. Similarly, with the roles of $n$ and $t$ exchanged, we can argue in case \ref{as:f:loct:pern} using Theorem~\ref{thm:cc}~\ref{lem:ccn}. The proofs of both results are variants of P.L.~Lions' original work as in \cite[Lemma~1.21]{Willem}. In the space-time setting a similar result can be found in \cite[Theorem~2.2]{henninger}.
\begin{theorem}\label{lem:ccot} Let $(u^j)_{j \in \N} \subset \mathcal{H}$ be bounded. Let $T>0$, $N \in \N$ and $q \in [2,\infty)$. Further assume
\begin{align*}
\lim_{j \to \infty}\sup_{(m_1,m_2)\in \Z^2}  \sum_{n\in\mathcal{Z}_{m_1}} \int_{\mathcal{R}_{m_2}}   |u^j_n(t)|^q \dt  =0 
\end{align*}
where we set $\mathcal{Z}_{m_1}=\{m_1N, m_1N+1, \ldots, (m_1+1)N-1 \}$ and $\mathcal{R}_{m_2} = [m_2 T, (m_2+1)T]$. 
Then $u^j \to 0$ in $L^{\tilde{q}}(\R; \ell^{\tilde q}(\Z))$ as $j \to \infty$ for all $\tilde{q} \in (2,\infty)$.
\end{theorem}

\begin{proof}
It is enough to show the result for $q=2$ since the case $q>2$ follows from $q=2$ by H\"older's inequality. Likewise, once we have shown the result for $\tilde{q} \in(2,3)$ the remaining cases $\tilde q\in [3,\infty)$ follow by Hölder's inequality and the boundedness of the embedding from Theorem~\ref{theom:emb}. 
In general, for $2<s<r<4$ and $\theta \in (0,1)$ being the solution of $\frac{1}{s}=\frac{\theta}{r}+\frac{1-\theta}{2}$, Hölder's inequality yields
\begin{align*}
    \|u\|_{L^s(\mathcal{R}; \ell^s(\mathcal{Z}))}^s \leq \|u\|_{L^r(\mathcal{R};\ell^r(\mathcal{Z}))}^{\theta s} \|u\|_{L^2(\mathcal{R};\ell^2(\mathcal{Z}))}^{(1-\theta)s}
\end{align*}
whenever $\mathcal{R}\subset \R$ is an interval and $\mathcal{Z}\subset \Z$. Now, for $r \in (2, 4)$ and $s=\frac{4(r-1)}{r}$, we have $2<s<r<4$, $2<s<3$ and $\theta=\frac{2}{s}$ so that we can use the above inequality to get  
\begin{align*}
    \|u\|_{L^s(\R; \ell^s(\Z))}^s&=\sum_{(m_1,m_2)\in \Z^2} \| u \|_{L^s(\mathcal{R}_{m_2}; \ell^s(\mathcal{Z}_{m_1}))}^s \\
    &\leq \sup_{(m_1,m_2)\in \Z^2} \| u \|_{L^2(\mathcal{R}_{m_2}; \ell^2(\mathcal{Z}_{m_1}))}^{s-2} \sum_{(m_1,m_2)\in \Z^2} \| u \|_{L^r(\mathcal{R}_{m_2}; \ell^r(\mathcal{Z}_{m_1}))}^2 \\
    & \leq c \sup_{(m_1,m_2)\in \Z^2} \| u \|_{L^2(\mathcal{R}_{m_2}; \ell^2(\mathcal{Z}_{m_1}))}^{s-2} \sum_{(m_1,m_2)\in \Z^2} \| u \|_{H^1(\mathcal{R}_{m_2}; \ell^2(\mathcal{Z}_{m_1}))}^2 
\end{align*}
where we have used the Sobolev Embedding Theorem and the embedding of $\ell^2$ into $\ell^r$.
Altogether we obtain 
\begin{align*}
     \|u\|_{L^s(\R; \ell^s(\Z))}^s\leq c \| u \|_{\mathcal{H}}^2 \sup_{(m_1,m_2)\in \Z^2} \| u \|_{L^2(\mathcal{R}_{m_2}; \ell^2(\mathcal{Z}_{m_1}))}^{s-2}
\end{align*}
which implies the claim of the theorem.
\end{proof}

\begin{theorem}\label{thm:cc}  Let $(u^j)_{j \in \N} \subset \mathcal{H}$ be bounded. 
\begin{enumerate}[(i)]
    \item Let $T>0$, $q \in [2, \infty)$ and $\mathcal{Z}\subset \Z$. Further assume
\begin{align*}
\lim_{j \to \infty}\sup_{m\in \Z} \sum_{n \in \mathcal{Z}} \int_{\mathcal{R}_m}   |u^j_n(t)|^q \dt  =0 
\end{align*}
where $\mathcal{R}_m = [mT, (m+1)T]$. Then $u^j \to 0$ in $L^{\tilde{q}}(\R; \ell^{\tilde q}(\mathcal{Z}))$ as $j \to \infty$ for all $\tilde{q} \in (2, \infty)$.\label{lem:cct}
\item Let $N\in \N$, $q \in [2, \infty)$ and $\mathcal{R}\subset \R$ be an interval. Further assume
\begin{align*}
\lim_{j \to \infty}\sup_{m\in \Z} \sum_{n \in \mathcal{Z}_m} \int_{\mathcal{R}}   |u^j_n(t)|^q \dt  =0 
\end{align*}
where $\mathcal{Z}_m = \{mN, mN+1, \ldots, (m+1)N-1\}$. Then $u^j \to 0$ in $L^{\tilde{q}}(\mathcal{R}; \ell^{\tilde q}(\Z))$ as $j \to \infty$ for all $\tilde{q} \in (2, \infty)$.\label{lem:ccn}
\end{enumerate}
   
\end{theorem}
Since the proof follows the same line of arguments as the proof of Theorem~\ref{lem:ccot} we omit it.

\section{Proof of the main result}\label{sec:proofmain}
The main result of Theorem~\ref{thm:main} follows from Theorem~\ref{thm:groundstate} which is proved at the end of this section.

Before we give the idea of the variational method, let us explain that we have an $\langle \cdot,\cdot\rangle_{\mathcal{H}}$-orthogonal as well as an $\langle \cdot,\cdot\rangle_{L^2_d}$-orthogonal decomposition $\mathcal{H}= \mathcal{H}^+ \oplus \mathcal{H}^-$. To see this, we recall first that $0\not \in \sigma(L)$ due to \ref{as:specgap}. Then we may use the spectral resolution $(P_\lambda)_{\lambda\in\R}$ of the operator $L$ to introduce the orthogonal projections
\begin{align*}
    P^+: L_d^2(\R; \ell^2(\Z)) \to L_d^2(\R; \ell^2(\Z)), \quad u \mapsto u^+ \coloneqq P^+[u]=\int_0^\infty 1 \der P_\lambda[u] ,
\end{align*}
\begin{align*}
    P^-: L_d^2(\R; \ell^2(\Z)) \to L_d^2(\R; \ell^2(\Z)), \quad u \mapsto u^- \coloneqq P^-[u]=\int^0_{-\infty} 1 \der P_\lambda[u] .
\end{align*}

With the help of these projections we decompose $\mathcal{H}= \mathcal{H}^+ \oplus \mathcal{H}^-$, $\mathcal{H}^\pm  \coloneqq P^\pm \mathcal{H}$ and thus we get $\|u\|_{\mathcal{H}}^2=\|u^+\|_{\mathcal{H}}^2+\|u^-\|_{\mathcal{H}}^2$ as well as $b_L(u,u) =\| u^+\|_{\mathcal{H}}^2-\|u^-\|_{\mathcal{H}}^2$. As a result our functional $J$ can be written as
\begin{align*}
    J(u)=\frac{1}{2} \|u^+\|_{\mathcal{H}}^2 - \frac{1}{2} \| u^-\|_{\mathcal{H}}^2 - J_1(u). 
\end{align*}

The idea is to minimize $J$ on the set 
\begin{align*}
    \mathcal{M}\coloneqq \set{u \in \mathcal{H} \setminus \mathcal{H}^- : J'(u)[w]=0 \text{ for all } w \in \R u+\mathcal{H}^-}
\end{align*}
which is known as the Nehari--Pankov manifold. Originally, Nehari~\cite{nehari1,nehari2} established this method in the setting where $\mathcal{H}^-=\{0\}$ whereas Pankov~\cite{pankov} generalized this approach to an indefinite setting. Later, Szulkin and Weth \cite{szulkin_weth} elaborated this method further, which is the main reference for our arguments. A minimizer of $J$ on $\mathcal{M}$ is a called a \emph{ground state}. Under the assumption that $J_1'(u)[u]>0$ for $u\not =0$ (which is fulfilled when applied to our case of Theorem~\ref{thm:main}) the ground state is also a \emph{least-energy solution} of $J'(u)=0$, i.e., $J$ is minimial among all solutions of $J'(u)=0$.

In order to know that a minimizer of $J$ on $\mathcal{M}$ exists, we will apply the theory from \cite{szulkin_weth}. For its applicability several conditions need to hold. The following lemma guarantees assumptions (B1), (B2), (i), (ii) of Theorem 35 in \cite{szulkin_weth}.

\begin{lemma}\label{lem:vorfunkt} The following properties of $J_1$ and $\mathcal M$ hold.
    \begin{enumerate}[(i)]
        \item $J_1$ is weakly lower semicontinuous, 
        \begin{align*}
            J_1(0)=0 \quad \text{ and } \quad \frac{1}{2}J_1'(u)[u]>J_1(u)>0 \text{ for } u \neq 0. 
        \end{align*} \label{lem:vorfunkt:i}
        \item  $\lim_{u\to 0} \frac{J_1'(u)}{\|u\|_{\mathcal{H}}}=0$ and $\lim_{u\to 0} \frac{J_1(u)}{\|u\|^2_{\mathcal{H}}}=0$.\label{lem:vorfunkt:ii}
        \item For a weakly compact set $U\subset \mathcal{H}\setminus\{0\}$ we have $\lim_{s\to\infty} \frac{J_1(s u)}{s^2} = \infty$ uniformly with respect to $u\in U$.\label{lem:vorfunkt:iii}
        \item  For each $w\in\mathcal{H}\setminus \mathcal{H}^-$ let $\mathcal{H}(w)=\R_{\geq 0}w+ \mathcal{H}^-$. Then there exists a unique nontrivial critical point $m(w)$ of $J|_{\mathcal{H}(w)}$. Moreover, $m(w)\in \mathcal{M}$ is the unique global maximizer of $J|_{\mathcal{H}(w)}$ and $J(m(w))>0$.\label{lem:vorfunkt:iv}
    \end{enumerate}
\end{lemma}

\begin{proof}
(i) If $u^j \rightharpoonup u \in \mathcal{H}$ then the boundedness of $B:\mathcal{H}\to\mathcal{H}$ from Lemma~\ref{lem:Bbounded} yields $Bu^j \rightharpoonup Bu \in \mathcal{H}$ as $j\to \infty$. By the locally compact embedding, cf.~Theorem~\ref{theom:emb}, we have (up to a subsequence) that $Bu^j\to Bu$ pointwise a.e. on $\Z \times \R$ as $j \to \infty$.  Since the mapping $s \mapsto F(n,t,s)d(t)$ is continuous and non-negative for all $(n,t) \in \Z \times \R$ the weak lower-semicontinuity of $J_1$ follows from Fatou's lemma. The statement $\frac{1}{2} J_1'(u)[u]> J_1(u)>0$ for all $u\in \mathcal{H}\setminus\{0\}$ then follows from Lemma~\ref{lem:est_on_f} (i) and the injectivity of $B$.

(ii) Let $\varepsilon>0$. Recall from Lemma~\ref{lem:est_on_f} (iv) that $|\tilde f(u)|\leq \varepsilon |u|+ C_\varepsilon |u|^p \in L^2(\R;\ell^2(\Z))$ for $u\in\mathcal{H}$. Since $\nabla_{L^2_d} J_1(u) = B^\ast\tilde f(Bu)$ we have that
\begin{align*}
     \norm{J_1'(u)}_{\mathrm{op}, \mathcal{H}\to \R}  \leq c_1 \norm{\nabla_{L^2_d} J_1(u)}_{L^2_d} &\leq c_1\sqrt{\norm{d}_{L^\infty}} \norm{B}_{\mathrm{op},L^2}\norm{\tilde{f}(Bu)}_{L^2} \\
     &\leq c_1 c_2\sqrt{\norm{d}_{L^\infty}} \norm{B}_{\mathrm{op},L^2} \sqrt{2} ( \varepsilon \norm{B}_{\mathrm{op},\mathcal{H}} \norm{u}_{\mathcal{H}}+ C_\varepsilon \norm{B}_{\mathrm{op},\mathcal{H}}^p \norm{u}_{\mathcal{H}}^p )
\end{align*}
where $c_1, c_2>0$ estimate embedding constants from Theorem~\ref{theom:emb}. Because this holds true for every $\varepsilon>0$, this proves the first claim. Similarly one can prove the second claim.

(iii) The statement means that $\lim_{s\to\infty} \inf_{u\in U} \frac{J_1(su)}{s^2}=\infty$. If this were not the case then there exists a sequence $s_j\to \infty$ and $(u^j)_{j\in\N}\subset U$ such that $\sup_j \frac{J_1(s_ju^j)}{s_j^2}<\infty$. Up to a subsequence we may assume $u^j\rightharpoonup u\not= 0$ and $Bu^j\rightharpoonup Bu\not= 0$ in $\mathcal{H}$ as $j\to \infty$. Let us select a measurable set $\Omega\subset \R\times\Z$ with fibers $\Omega_n \coloneqq \{t\in \R: (t,n)\in \Omega\}$ and $0<|\Omega| \coloneqq \sum_{n\in \Z} |\Omega_n|<\infty$ (using $|\cdot|$ for the 1-dimensional Lebesgue measure) such that 
$$
0<a \coloneqq \left| \sum_{n\in\Z}\int_{\Omega_n} (Bu)_n(t) d(t)\der t\right|
$$
and (by weak convergence) for $\delta>0$ and sufficiently large $j$
\begin{align*}
\frac{a}{2} &\leq  \left| \sum_{n\in \Z}\int_{\Omega_n} (Bu^j)_n(t) d(t)\der t\right| \\
& \leq \sum_{n\in \Z} \int_{\Omega_n \cap \{|(Bu^j)_n(t)|< \delta\}} |(Bu^j)_n(t)| d(t)\der t + \sum_{n\in \Z} \int_{\Omega_n \cap \{|(Bu^j)_n(t)|\geq\delta\}} |(Bu^j)_n(t)| d(t)\der t.
\end{align*}
For $\delta>0$ small enough we get 
$$
\frac{a}{4} \leq   |\Omega|^\frac{1}{2}  \|d\|_\infty^{\frac12} \left(\sum_{n\in \Z} \int_{\Omega_n \cap \{|(Bu^j)_n(t)|\geq \delta\}} |(Bu^j)_n(t)|^2 d(t)\der t\right)^\frac{1}{2}.
$$

On the set $\Omega_n \cap \{|(Bu^j)_n(t)|\geq \delta\}$ we have for every $K>0$ and large enough $j$ by assumption~\ref{as:f4}
$$
\frac{F(n,t, s_j (Bu^j)_n(t))}{s_j^2} \geq K |(Bu^j)_n(t)|^2 
$$
so that 
\begin{align*}
\frac{J_1(s_j u^j)}{s_j^2} &\geq \sum_{n\in\Z} \int_{\Omega_n \cap \{|(Bu^j)_n(t)|\geq \delta\}} \frac{F(n,t, s_j(Bu^j)_n(t))}{s_j^2} d(t)\der t \\
& \geq 
K \int_{\Omega_n \cap \{|(Bu^j)_n(t)|\geq \delta\}} |(Bu^j)_n(t)|^2 d(t) \der t \geq \frac{a^2K}{16|\Omega| \|d\|_\infty}
\end{align*}
for large enough $j$. Since $K>0$ was arbitrary this contradicts the assumption $\sup_j \frac{J_1(s_ju^j)}{s_j^2}<\infty$.

(iv) The proof is based on the proof of Proposition 39 in \cite{szulkin_weth} and consists of three steps which will be proved below. Recall that $w\in \mathcal{H}\setminus\mathcal{H}^-$.\\
\textit{Step 1:} $u\in \mathcal{M} \Leftrightarrow u$ is a critical point of $J|_{\mathcal{H}(u)}$. If so, then $u$ is the unique global maximizer of $J|_{\mathcal{H}(u)}$.\\
\textit{Step 2:} $\mathcal{H}(w)\cap \mathcal{M}\not = \emptyset$ \\
\textit{Step 3:} The restricted functional $J|_{\mathcal{H}(w)}$ has a unique global maximizer $m(w)$ with $J(m(w))>0$ and no other critical point.

The proof of the equivalence in \textit{Step 1} consists in nothing else then the inspection of the definition of $\mathcal{M}$. To see that $u$ is the unique global maximizer of $J|_{\mathcal{H}(u)}$ we use the pointwise inequality from \cite[Lemma 38]{szulkin_weth}
$$
f(n,t,s)[\xi(\frac{\xi}{2}+1)s+(1+\xi)\eta]+F(n,t,s)-F(n,t,(1+\xi)s+\eta)<0
$$ 
for $s,\eta,\xi\in \R, \xi s+\eta\not =0$ and $\xi\geq-1$. For $v\in \mathcal{H}^-$ and $z:= \xi u+v$, $\xi\geq -1$ it is shown in \cite[Lemma 38]{szulkin_weth} that
this implies $J(u+z)<J(u)$ provided $z\not =0$.
Since the elements $u+z$ cover all of $\mathcal{H}(u)$ this shows that $u$ is the unique global maximizer of $J|_{\mathcal{H}(u)}$. The proof of \textit{Step 2} follows the same lines as the proof of Proposition 39 (i) in \cite{szulkin_weth}.  To see \textit{Step 3}, take $u\in \mathcal{M}\cap \mathcal{H}(w)$, cf. \textit{Step 2}. By \textit{Step 1}, $u$ is the unique global maximizer of $J|_{\mathcal{H}(u)}= J|_{\mathcal{H}(w)}$. We write $u\eqqcolon m(w)$ and since $m(w)\in \mathcal{M}$ we get 
$$
J(m(w)) = J(m(w)) - \frac{1}{2}\underbrace{J'(m(w))m(w)}_{=0} = \frac{1}{2}J_1'(m(w))m(w)- J_1(m(w))>0
$$
by (i). Finally, if $z$ is any critical point of $J|_{\mathcal{H}(w)}$ by \textit{Step 1} necessarily $z=m(w)$.
\end{proof}

In the next lemma we show that also assumption (iii) in Theorem 35 in \cite{szulkin_weth} holds true when we are in case \ref{as:f:locnt} of Theorem~\ref{thm:main}. As a consequence, existence of a minimizer of $J$ on $\mathcal{M}$ immediately follows.

\begin{lemma}\label{lem:complcont}
    If $f$ satisfies assumption~\ref{as:f:locnt}, then the mapping $J_1':\mathcal{H} \to \mathcal{H}', u \mapsto J_1'(u)$ is completely continuous.
\end{lemma}

\begin{proof} Let $(u^j)_{j \in \N}$ be a sequence in $\mathcal{H}$ such that $u^j \rightharpoonup u \in \mathcal{H}$ as $j \to \infty$. 
Using $\nabla_{L^2_d} J_1'(u)=B^\ast \tilde f(Bu)$ and a suitable embedding constant $c$ from Theorem~\ref{theom:emb} we have
\begin{align*}
    \norm{J_1'(u^j)-J_1'(u)}_{\mathrm{op}, \mathcal{H}\to \R} &\leq   c\norm{\nabla_{L^2_d} J_1(u^j)-\nabla_{L^2_d} J_1(u)}_{L^2_d(\R;\ell^2(\Z))} \\
    & \leq c\norm{B}_{\mathrm{op},\ell^2(\Z)} \norm{ \tilde f(Bu^j)-\tilde f(Bu)}_{L^2_d(\R;\ell^2(\Z))} 
\end{align*}    
and we will show that $\norm{ \tilde f(Bu^j)-\tilde f(Bu)}_{L^2_d(\R;\ell^2(\Z))}\to 0$ as $j\to\infty$. For a bounded interval $\mathcal{R}\subset\R$ and a bounded set $\mathcal{Z}\subset\Z$ we estimate
\begin{align*}
    \norm{\tilde f(Bu^j)-\tilde f(Bu)}_{L^2(\R; \ell^2(\Z))}^2\leq  & \norm{\tilde f(Bu^j)-\tilde f(Bu)}_{L^2(\mathcal{R}; \ell^2(\mathcal{Z}))}^2 \\
    & +\norm{\tilde f(Bu^j)-\tilde f(Bu)}_{L^2(\mathcal{R}^c; \ell^2(\mathcal{\Z}))}^2
    & + \norm{\tilde f(Bu^j)-\tilde f(Bu)}_{L^2(\R; \ell^2(\mathcal{Z}^c))}^2.
\end{align*}
For $\rho >0$ and $\varepsilon>0$ there exists by Lemma~\ref{lem:est_on_f} a positive, bounded function $h_\varepsilon(n,t)$ such that $\lim_{|(n,t)| \to \infty}h_\varepsilon
(n,t)=0$ and
\begin{align*}
    |f(n,t,s)|^2 \leq \varepsilon |s|^2 + h_\varepsilon(n,t) |s|^{2(p+\rho)} \text{ for all } (n,t,s) \in \Z \times \R \times \R .
\end{align*}
By choosing $\mathcal R$ and $\mathcal Z$ to be sufficiently large we can achieve that $|f(n,t,s)|^2 \leq \varepsilon (|s|^2+ |s|^{2(p+\rho)})$ for all $(n,t) \in (\R\times\mathcal{Z}^c) \cup (\mathcal{R}^c\times\Z)$ and all $s\in\R$.
A direct application of this estimate leads to
\begin{align}
     \norm{\tilde f(Bu^j)-\tilde f(Bu)}_{L^2(\mathcal{R}^c;\ell^2(\Z))}^2 \nonumber
     \leq & \varepsilon ( \norm{Bu^j}_{L^2(\mathcal{R}^c;\ell^2(\Z))}^2+ \norm{Bu}_{L^2((\mathcal{R}^c;\ell^2(\Z))}^2) \nonumber\\
     &+\varepsilon(\norm{Bu^j}_{L^{2(p+\rho)}(\mathcal{R}^c;\ell^{2(p+\rho)}(\Z))}^{2(p+\rho)})
     +\norm{Bu}_{L^{2(p+\rho)}(\mathcal{R}^c;\ell^{2(p+\rho)}(\Z))}^{2(p+\rho)} ) \label{eq:est_diffB}\\
     \leq & \varepsilon \tilde c (\norm{Bu^j}_\mathcal{H}^2+ \norm{Bu}_\mathcal{H}^2+ \norm{Bu^j}_\mathcal{H}^{2(p+\rho)}+ \norm{Bu}_\mathcal{H}^{2(p+\rho)}) \nonumber \\
      \leq & \varepsilon C \nonumber
\end{align}
where $\tilde c$ estimates the embedding constants from Theorem~\ref{theom:emb} and $C$ additionally takes into account the boundedness of the sequence $(Bu^j)_{j \in \N}$ in $\mathcal{H}$. A similar estimate holds for $\norm{\tilde f(Bu^j)-\tilde f(Bu)}_{L^2(\R; \ell^2(\mathcal{Z}^c))}^2$. Lastly, it remains to analyze the difference $\tilde f(Bu^j)-\tilde f(Bu)$ on the bounded domain $\mathcal{R}\times\mathcal{Z}$.
Since $u^j \rightharpoonup u$ in $\mathcal{H}$ and $Bu^j \rightharpoonup Bu$ in $\mathcal{H}$ we use the compactness of the Nemytskii operator from Lemma~\ref{lem:est_on_f}(iv) to obtain $\norm{\tilde f(Bu^j)-\tilde f(Bu)}_{L^2(\mathcal{R};\ell^2(\mathcal{Z}))}^2\to 0$ as $j\to \infty$. Together with \eqref{eq:est_diffB} this finishes the proof.
\end{proof}

The proof of Lemma~\ref{lem:complcont} is based on the interaction of the localization property of $f$ with respect to space and time and local compactness of the embedding from Theorem~\ref{theom:emb}.
In cases \ref{as:f:pernt}, \ref{as:f:locn:pert} and \ref{as:f:loct:pern} of Theorem~\ref{thm:main} the function $f$ cannot directly recover the lack of global compactness of the embedding. Nevertheless, since we have imposed an additional structure on \eqref{eq:main}, namely periodicity, we can bypass this issue.

\begin{lemma}\label{lem:PSbounded}
 Let $f$, $M$, $B$ and $d$ be as in case \ref{as:f:pernt}, \ref{as:f:locn:pert} or \ref{as:f:loct:pern}  of Theorem~\ref{thm:main}. Then any Palais--Smale sequence $(u^j)_{j \in \N}$ of $J|_{\mathcal{M}}$ is bounded.
\end{lemma}

\begin{proof}
    The following proof is adapted from \cite[Lemma~5.5]{hr}. We assume  for contradiction that $(u^j)_{j \in \N} \subset \mathcal{M}$ is an unbounded Palais--Smale sequence for $J$ so that for a subsequence (again denoted by $u^j$) we have $\norm{u^j}_{\mathcal{H}}\to \infty$ as $j \to \infty$. The sequence $v^j\coloneqq\frac{u^j}{\norm{u^j}_{\mathcal{H}}}$ is bounded and hence there exists a subsequence (again denoted by $v^j$) and $v \in \mathcal{H}$ such that $v^j \rightharpoonup v$ as $j \to \infty$.
 Since $J(u^j)=J(\norm{u^j}_{\mathcal{H}}v^j )$ and  $u^j$ is in $\mathcal{M}$ we know that $J(u^j)\geq 0$. This leads to the estimate
\begin{align}\label{eq:PS_estimate}
    0 \leq \frac{J(u^j)}{\norm{u^j}_{\mathcal{H}}^2} = \frac{1}{2} \norm{v^{j,+}}_{\mathcal{H}}^2-\frac{1}{2} \norm{v^{j,-}}_{\mathcal{H}}^2 - \frac{J_1(\norm{u^j}_{\mathcal{H}}v^j)}{\norm{u^j}_{\mathcal{H}}^2} .
\end{align}
 Moreover, from $J_1 \geq 0$ and \eqref{eq:PS_estimate} we infer that $\norm{v^{j,-}}_{\mathcal{H}}^2 \leq \norm{v^{j,+}}_{\mathcal{H}}^2$. In particular, since $\norm{v^{j,+}}_{\mathcal{H}}^2+\norm{v^{j,-}}_{\mathcal{H}}^2=1$ we have
\begin{align}\label{eq:normvj+}
\frac{1}{2}\leq \norm{v^{j,+}}_{\mathcal{H}}^2\leq 1 .    
\end{align}

For the rest of the proof we distinguish between the cases \ref{as:f:pernt}, \ref{as:f:locn:pert} and \ref{as:f:loct:pern}. 

\underline{Case \ref{as:f:pernt}}:
 First we show that $v^{j,+} \not \to 0$ in $L^{p+1}(\R;\ell^{p+1}(\Z))$ as $j \to \infty$. Since for every $\varepsilon >0$ there exists $C_\varepsilon >0$ such that $|F(n,t,s)| \leq \varepsilon s^2+C_\varepsilon |s|^{p+1}$ (see Lemma~\ref{lem:est_on_f}~\ref{lem:est_on_f:gen_est}) 
 we estimate $0 \leq J_1(u) \leq c( \varepsilon \norm{u}_{\mathcal{H}}^2+ C_\varepsilon  \norm{u}_{L^{p+1}(\R; \ell^{p+1}(\Z))}^{p+1} )$ for every $u \in \mathcal{H}$ with $c>0$ taking into account an embedding constant from Theorem~\ref{theom:emb} and operator norms of $B$ from Lemma~\ref{lem:Bbounded}.
From this estimate combined with Lemma~\ref{lem:vorfunkt}~\ref{lem:vorfunkt:iv} and \eqref{eq:normvj+} we obtain for any $\tau>0$ that 
\begin{align}\label{eq:contra}
    J(u^j)\geq J(\tau v^{j,+})&=\frac{\tau^2}{2} \norm{v^{j,+}}_{\mathcal{H}}^2-J_1(\tau v^{j,+}) \\
    &\geq \frac{\tau^2}{4}- \varepsilon  c  \tau^2 \norm{v^{j,+}}_{\mathcal{H}}^2- C_\varepsilon  c|\tau|^{p+1}\norm{v^{j,+}}_{L^{p+1}(\R; \ell^{p+1}(\Z))}^{p+1}\nonumber \\
    &\geq \frac{\tau^2}{8}  - C_\varepsilon c|\tau|^{p+1}\norm{v^{j,+}}_{L^{p+1}(\R; \ell^{p+1}(\Z)))}^{p+1}    \nonumber
\end{align}
for all $j \in \N$  by an appropriate choice of $\varepsilon$.  Since $(u^j)_{j\in \N}$ being a Palais--Smale sequence implies that the left hand side of \eqref{eq:contra} is bounded but $\tau$ can be arbitrary, we conclude that $v^{j,+} \not\to 0$ in $L^{p+1}(\R;\ell^{p+1}(\Z))$ for $j \to \infty$ as claimed.

Next, we construct a new sequence $(w^j)_{j \in \N}$ based on $(v^j)_{j \in \N}$. As we have just shown, $v^{j,+} \not \to 0$ in $L^{p+1}(\R;\ell^{p+1}(\Z))$ so that by Theorem~\ref{lem:ccot} there exists $\delta>0$, a sequence of translations $(m_1^j,m_2^j)_{j \in \N} \in \Z^2$ and a subsequence of $v^{j,+}$ (again denoted by $v^{j,+}$) with the property
\begin{align*}
    \sum_{n=m_1^j N}^{(m_1^j+1)N-1} \int_{[m_2^j T, (m_2^j+1)T]} |v_n^{j,+}(t)|^{2} \dt \geq \delta >0
\end{align*}
for all $j \in \N$. With this at hand we introduce the sequence $w^j_n(t) \coloneqq v^j_{n+m_1^j N}(t+m_2^j T)$. 
The periodicity assumption \ref{as:f:pernt} and Lemma~\ref{lem:operatorcommute}(ii) imply that $\norm{w^j}_{\mathcal{H}}=\norm{v^j}_{\mathcal{H}}$ and
\begin{align}\label{eq:shifted}
    \sum_{n=0}^{N-1} \int_{[0,T]}|w_n^{j,+}(t)|^2 \dt = \sum_{n=m_1^jN}^{(m_1^j+1)N-1} \int_{[m_2^jT, (m_2^j+1)T]} |v_n^{j,+}(t)|^2 \dt \geq \delta >0
\end{align}
holds for all $j \in \N$. Further, there exists $w \in \mathcal{H}$ such that (up to a subsequence) $w^j \rightharpoonup w $ in $\mathcal{H}$ as $j\to \infty$. In particular $w^{j,+} \rightharpoonup w^+$ in $\mathcal{H}$ and $w^{j,+} \to w^+$ in $L^2([0,T]; \ell^2(\{0,\ldots, N-1\})$ as $j \to \infty$ by the locally compact embedding result from Theorem~\ref{theom:emb}. From \eqref{eq:shifted} we get $w^+\not =0$ thus $w\not =0$.
On the other hand, by our assumption \ref{as:f:pernt} and Lemma~\ref{lem:operatorcommute} it follows that $J(u^j)=J(\norm{u^j}_{\mathcal{H}}v^j)=J(\norm{u^j}_{\mathcal{H}}w^j)$ and together with \eqref{eq:PS_estimate} this leads to 
\begin{align} \label{eq:final_contra}
      0 \leq \frac{J(u^j)}{\norm{u^j}_{\mathcal{H}}^2} = \frac{1}{2} \norm{w^{j,+}}_{\mathcal{H}}^2-\frac{1}{2} \norm{w^{j,-}}_{\mathcal{H}}^2 - \frac{J_1(\norm{u^j}_{\mathcal{H}}w^j)}{\norm{u^j}_{\mathcal{H}}^2}.
\end{align}
If we take the limit $j \to \infty$ and apply Lemma~\ref{lem:vorfunkt}~\ref{lem:vorfunkt:iii} to the weakly compact set $U=\{ w^j : j \in \N\}\cup\{w\} \not \ni 0$, then the right hand side of \eqref{eq:final_contra} tends to $- \infty$,  which is impossible. This contradiction shows that any Palais--Smale sequence $(u^j)_{j \in \N}$ of $J$ in $\mathcal{M}$ must be bounded. This finishes the proof in case \ref{as:f:pernt}. \\

\underline{Case \ref{as:f:locn:pert}}:
We proceed similarly as in case \ref{as:f:pernt}. First we show that for $\rho >0$ there exists a bounded set $\mathcal{Z}_0 \subset \Z$ such that  $\liminf_{j\in \N} \norm{v^{j,+}}_{L^{p+1+\rho}(\R; \ell^{p+1+\rho}(\mathcal{Z}_0))}>0$. To see this, we take an arbitrary $\varepsilon>0$ and estimate
\begin{align*}
    0 \leq J_1(u) \leq \varepsilon c \norm{u}_{\mathcal{H}}^2+ \int_\R \sum_{n \in \Z} h_{\varepsilon}(n,t) |(Bu)_n(t)|^{p+1+\rho}d(t) \der t
\end{align*}
for all $u \in \mathcal{H}$ where the function $h_\varepsilon$ with $h_\varepsilon(n,t) \to 0$ as $|n| \to \infty$ uniformly in $t$ arises from Lemma~\ref{lem:est_on_f}~\ref{lem:est_on_f:eps} (b) and $c$ stems from Theorem~\ref{theom:emb} and the boundedness of $B$.
For any $\tau>0$ we then obtain from the above estimate combined with Lemma~\ref{lem:vorfunkt}~\ref{lem:vorfunkt:iv} and \eqref{eq:normvj+} 
\begin{align}\label{eq:contrab}
    J(u^j) \geq J(\tau v^{j,+}) &=\frac{\tau^2}{2}\norm{v^{j,+}}_{\mathcal{H}}^2-J_1(\tau v^{j,+}) \nonumber \\
    &\geq \frac{\tau^2}{4}-\varepsilon c \tau^2 \norm{v^{j,+}}_{\mathcal{H}}^2 -|\tau|^{p+1+\rho} \int_{\R} \sum_{n \in \Z} h_{\varepsilon}(n,t)|(Bv^{j,+})_n(t)|^{p+1+\rho} d(t) \der t \nonumber \\
    &\geq \frac{\tau^2}{8} - |\tau|^{p+1+\rho} \int_{\R} \sum_{n \in \Z} h_{\varepsilon}(n,t)|(Bv^{j,+})_n(t)|^{p+1+\rho} d(t) \der t
\end{align}
for all $j \in \N$ and $\varepsilon$ sufficiently small.
Now, let us suppose for contradiction that $v^{j,+} \to 0$ in $L^{p+1+\rho}(\R; \ell^{p+1+\rho}(\mathcal{Z}))$ for every bounded $\mathcal{Z}\subset \Z$.
Together with the localization property of $h_\varepsilon$ this implies $\int_{\R} \sum_{n \in \Z} h_{\varepsilon}(n,t)|Bv^{j,+}|^{p+1+\rho} \der t \stackrel{j\to \infty}{\to} 0$.
This leads to a contradiction since for the Palais--Smale sequence $(u^j)_{j \in \N}$ the left hand side of \eqref{eq:contrab} is bounded whereas the right hand side becomes arbitrarily large for $|\tau|\to \infty$. Thus we have shown the existence of a bounded set $\mathcal{Z}_0\subset \Z$ such that $\liminf_{j\in \N} \norm{v^{j,+}}_{L^{p+1+\rho}(\R; \ell^{p+1+\rho}(\mathcal{Z}_0))}>0$.

Next, in view of Theorem~\ref{thm:cc}~\ref{lem:cct}, we obtain $\delta>0$ and a sequence of translations $(m^j)_{j\in \Z}$ such that (up to a subsequence) the sequence $w_n^j(t)\coloneqq v_n^j(t+m^jT)$ satisfies
\begin{align}\label{eq:shiftedseq}
     \int_{[0,T]} \sum_{n \in \mathcal{Z}_0} |w_n^{j,+}(t)|^2 \dt =  \int_{[m^jT, (m^j+1)T]} \sum_{n \in \mathcal{Z}_0} |v_n^{j,+}(t)|^2 \dt \geq \delta >0
\end{align}
for all $j \in \N$. Since the periodicity assumption \ref{as:f:locn:pert} implies $\norm{w^j}_{\mathcal{H}}=\norm{v^j}_{\mathcal{H}}$, there exists $w \in \mathcal{H}$ with $w^j \rightharpoonup w$ as $j \to \infty$ (up to a subsequence) from which we infer by local compactness (cf. Theorem~\ref{theom:emb}) and \eqref{eq:shiftedseq} that $w^+ \neq 0$ and hence $w \neq 0$. 
We finish the proof by arguing identically as for \eqref{eq:final_contra} in Case~\ref{as:f:pernt} to get a contradiction.
Therefore, any Palais--Smale sequence $(u^j)_{j\in \N}$ of $J$ in $\mathcal{M}$ must be bounded.  \\
\underline{Case \ref{as:f:loct:pern}:} One can proceed as in case \ref{as:f:locn:pert} with the roles of $n$ and $t$ exchanged, in particular by making use of the properties of the function $h_\varepsilon$ from Lemma~\ref{lem:est_on_f}~\ref{lem:est_on_f:eps} (c) combined with the concentration compactness argument from Theorem~\ref{thm:cc}~\ref{lem:ccn}.
\end{proof}

The following density lemma is obvious by approximating infinite series with finite sums and $H^1(\R)$ functions by $C_c^\infty(\R)$ functions. 
\begin{lemma}\label{rem:4.2}
The set $V_0 \coloneqq \{v=(v_n)_{n\in \Z} \text{ with } v_n\in C_c^\infty(\R) \text { and } v_n=0 \text{ for almost all } n\in \Z\}$ is dense in $\mathcal{H}$.
\end{lemma}

The main Theorem~\ref{thm:main} follows from the next result.

\begin{theorem}\label{thm:groundstate} The functional $J$ admits a ground state $u\in \mathcal{H}$. Moreover $u\in H^2(\R;\ell^2(\Z))\subset C^1_0(\R;\ell^2(\Z))$ and $u$ is a rogue wave solution of \eqref{eq:main}.
\end{theorem}

\begin{proof} For the existence of a ground state we use Theorem 35 in \cite{szulkin_weth}. 
By Lemma~\ref{lem:vorfunkt} the conditions $(B_1), (B_2), (i)$ and $(ii)$ in Theorem 35 are satisfied. 
If we assume case \ref{as:f:locnt}, then also condition $(iii)$ in Theorem 35 holds by Lemma~\ref{lem:complcont} and we directly obtain  a ground state of $J$ as a minimizer of $J|_\mathcal{M}$.

In cases \ref{as:f:pernt}, \ref{as:f:locn:pert} and \ref{as:f:loct:pern} condition $(iii)$ in Theorem 35 does not hold. Nevertheless, an inspection of the proof of Theorem~35 shows that the assumptions $(B_1), (B_2), (B_3)$ and $(i)$ still provide a minimizing Palais--Smale sequence and moreover that $(B_3)$ in fact follows from $(B_1)$, $(B_2)$, $(i)$ and $(ii)$. Therefore, as observed at the beginning of the proof, in our setting Lemma~\ref{lem:vorfunkt} is sufficient for the existence of a  minimizing Palais--Smale sequence of $J$ in $\calM$. As we shall see next, this still leads to a critical point of $J$. 

The following arguments are partly based on the proof of Theorem 5.6 in \cite{hr}. Let $(u^j)_{j \in \N}$ be a minimizing  Palais--Smale sequence of $J$ in $\calM$. Lemma~\ref{lem:PSbounded} guarantees that $(u^j)_{j \in \N}$ is bounded and thus there exist $u \in \mathcal{H}$  and a subsequence (again denoted by $(u^j)_{j \in \N}$) such that $u^j \rightharpoonup u$ in $\mathcal{H}$ and $J_0'(u^j)[v] \to J_0'(u)[v]$ as $j \to \infty$ for all $v \in \mathcal{H}$.
Further, by the local compactness of the Nemytskii operator from Lemma~\ref{lem:est_on_f} (iv) we have for all compact support functions $v\in V_0$
\begin{align*}
    J_1'(u^j)[v]=\sum_{n \in \Z} \int_\R B^*f(n,t,(Bu^j)_n(t)) v_n(t) d(t) \dt \to J_1'(u)[v] 
\end{align*} 
as $j \to \infty$.
By Lemma~\ref{rem:4.2} this is sufficient to conclude that $u$ is a critical point of $J$.

For the rest of the proof we distinguish between the cases \ref{as:f:pernt}, \ref{as:f:locn:pert} and \ref{as:f:loct:pern}.

\underline{Case \ref{as:f:pernt}}: Using that $(u^j)_{j\in \N}$ is bounded we can follow the proof of Lemma~\ref{lem:PSbounded}  (where $u^j$ takes the role of $v^j$) and obtain $u^{j,+} \not \to 0$ in $ L^{p+1}(\R;\ell^{p+1}(\Z))$ as $j \to \infty$. Hence, using Theorem~\ref{lem:ccot}, we find $\delta>0$, a sequence of translations $(m_1^j,m_2^j)_{j \in \N} \in \Z^2$ and a subsequence of $u^{j}$ (again denoted by $u^j$) such that
\begin{align} \label{eq:l2_lowerbound}
    \sum_{n=m_1^j N}^{(m_1^j+1)N-1} \int_{[m_2^j T, (m_2^j+1)T]} |u_n^{j,+}(t)|^{2} \dt \geq \delta >0
\end{align}
for all $j \in \N$.
If we define $v_n^j(t)\coloneqq  u^j_{n+m_1^j N}(t+m_2^j T)$ then the periodicity assumption and Lemma~\ref{lem:operatorcommute} allow to conclude $\norm{u^j}_\mathcal{H}=\norm{v^j}_\mathcal{H}$ and $J(u^j)=J(v^j)$ for all $j \in \N$ as well as $J'(v^j)\stackrel{j\to\infty}{\to} 0$. 
In particular, $(v^j)_{j \in \N}$ is again a bounded minimizing Palais--Smale sequence for $J$ so that there exists $v \in \mathcal{H}$ with $v^j \rightharpoonup v$ in $\mathcal{H}$ (up to a subsequence).
By the locally compact $L^2$-embedding from Theorem~\ref{theom:emb}, \eqref{eq:l2_lowerbound} yields $v^+\not = 0$. From the observation at the beginning of this proof we get $J'(v)=0$ and therefore $v \in \mathcal{M}$.

Lastly, we show that $v$ minimizes $J$ on $\calM$.
Since $v \in \calM$ we obviously have $J(v) \geq \inf_\calM J$.
For the reverse inequality, we use that $v^j \to v$ pointwise a.e. (for a subsequence) and $\frac{1}{2}f(n,t,s)s-F(n,t,s)\geq 0$ from Lemma~\ref{lem:est_on_f} (i) together with Fatou's Lemma 
\begin{align*}
\inf_\calM J &= \lim_{j \to \infty} J(v^j) - \frac{1}{2} J'(v^j)[v^j]=  \lim_{j \to \infty} \sum_{n \in \Z} \int_{\R} \frac{1}{2} f(n,t,(Bv^j)_n(t))(Bv^j)_n(t) -F(n,t,(Bv^j)_n(t)) \dt \\
&\geq \sum_{n \in \Z} \int_{\R} \frac{1}{2} f(n,t,(Bv)_n(t))(Bv)_n(t) -F(n,t,(Bv)_n(t)) \dt  = J(v) - \frac{1}{2} J'(v)[v]=J(v) .
\end{align*}

\underline{Case \ref{as:f:locn:pert}}: We proceed similarly as in case \ref{as:f:pernt}. Since $(u^j)_{j\in \N}$ is bounded, following the lines of the proof of Lemma~\ref{lem:PSbounded}, we see that for $\rho>0$ there exists a bounded set $\mathcal{Z}_0 \subset \Z$ such that $u^{j,+} \not \to 0$ in $L^{p+1+\rho}(\R;\ell^{p+1+\rho}(\mathcal{Z}_0))$.
By Theorem~\ref{thm:cc}~\ref{lem:cct} we obtain a $\delta>0$ and a sequence of translations $(m^j)_{j\in \N}$ such that (up to a subsequence) the sequence $v_n^j(t)\coloneqq u_j^j(t+m^jT)$ satisfies
\begin{align}\label{eq:l_2_lowerboundB}
    \int_{[0,T]}\sum_{n\in\mathcal{Z}_0} |v_n^{j,+}(t)|^2 \der t = \int_{[m^jT,(m^j+1)T]} \sum_{n \in \mathcal{Z}_0}|u_n^{j,+}(t)|^2 \der t \geq \delta >0
\end{align}
for all $j \in \N$.
By the periodicity assumption, the sequence $(v^j)_{j\in\N}$ is again a bounded Palais--Smale sequence for $J$. Hence, there exists $v \in \mathcal{H}$ such that $v^j \rightharpoonup v$ (up to a subsequence) as $j \to \infty$.  From \eqref{eq:l_2_lowerboundB} and the locally compact $L^2$ embedding (Theorem~\ref{theom:emb}) we get $v^+\neq 0$. Since $J'(v)=0$ as observed at the beginning of the proof, $v \in \mathcal{M}$.
The fact that $v$ minimizes $J$ on $\mathcal{M}$ follows exactly as in case \ref{as:f:pernt}.

\underline{Case \ref{as:f:loct:pern}:} We can proceed as in case \ref{as:f:locn:pert} with the roles of $n$ and $t$ exchanged and by making use of the concentration compactness argument from Theorem~\ref{thm:cc}~\ref{lem:ccn}. 

\medskip
In all cases \ref{as:f:locnt}, \ref{as:f:pernt}, \ref{as:f:locn:pert}, \ref{as:f:loct:pern} we found a critical point $u$ of $J$ as a ground state and hence a weak rogue wave solution of \eqref{eq:main} in the sense of Definition~\ref{def:weak_sol}. Let us check that $-Mud(\cdot)+ B^\ast\tilde f(Bu)d(\cdot)\in L^2(\R;\ell^2(\Z))$. First, $Mud(\cdot) \in L^2(\R;\ell^2(\Z))$ by boundedness of $M:L^2(\R;\ell^2(\Z))\to L^2(\R;\ell^2(\Z))$. Second, $B^\ast\tilde f(Bu)d(\cdot)\in L^2(\R;\ell^2(\Z))$ by boundedness of $B:\mathcal{H}\to \mathcal{H}$, $B^\ast:L^2(\R;\ell^2(\Z))\to L^2(\R;\ell^2(\Z))$ and the mapping property $\tilde f: \mathcal{H}\to L^2(\R;\ell^2(\Z))$ of the Nemytskii-operator from Lemma~\ref{lem:est_on_f}(iv). Hence, the definition of a weak solution implies that $u\in H^2(\R;\ell^2(\Z))$. Finally, the property $H^2(\R;\ell^2(\Z))\subset C^1_0(\R;\ell^2(\Z))$ can be derived from Theorem~\ref{theom:emb}.

Lastly, replacing $f$ by $-f$ in \eqref{eq:main} is equivalent to study the equation $-Lu=B^\ast f^\sharp(t,Bu)$ and to consider the functional $J=-J_0-J_1$. Since $L$ is indefinite we can proceed with the same variational arguments with the roles of $\mathcal{H}^+$ and $\mathcal{H}^-$ interchanged. The norm $\norm{\cdot}_{\mathcal{H}}$ does not change since it only depends on $|L|$.
\end{proof}

\begin{proof}[Proof of Corrollary~\ref{cor:breathers}]
Let $d$ and $f$ be $T${\color{blue}-}periodic in time. We consider \eqref{eq:main} on the space-time domain $\Z \times \T_{kT}$ where $\T_{kT}$ is the flat torus with period length $kT$ for some $k \in \N$. If we now choose $\mathcal{H}$ as $H^1(\T_{kT};\ell^2(\Z))$ then one can directly adapt the variational setting to $kT$-time periodic functions. In this case the variational method from \cite{szulkin_weth} applies in exactly the same way as before and shows the existence of nontrivial $kT$-time periodic solutions. The assumption analogous to \ref{as:specgap} on the spectrum of $L$, which is needed now, reads as $\sigma\left( M \right) \cap \sigma\left( -\frac{1}{d(t)} \dtsquare \right)_{L^2(\T_{kT})}=\emptyset$ and it is true for any $k\in \N$ since $\sigma\left( -\frac{1}{d(t)} \dtsquare\right)_{L^2(\T_{kT})}\subset \sigma\left( -\frac{1}{d(t)} \dtsquare\right)_{L^2(\R)}$ by Floquet-Bloch theory, see \cite{Eastham}.
\end{proof}

	\appendix
\section{Examples for $M$ and $d$}\label{sec:examples}
In this section we provide examples for the operator $M$ and the function $d$ and show that they satisfy the assumptions  \ref{as:M}--\ref{as:specgap}.
We begin with an example for the operator $M$.
 \begin{lemma}\label{lem:spec_disc_lap_q}
Let $M=- \Delta_1 +q$ with $q\in \ell^\infty(\Z)$, $q\geq 0$. Then $M$ satisfies assumption \ref{as:M} and $\sigma(M)\subset[\inf_\Z q, \sup_\Z q +4]$.
\end{lemma}
 \begin{proof}
From \cite[Chapter 4.1.4]{borht} we know $\sigma(-\Delta_1)=[0,4]$. Since for any linear, bounded and self-adjoint operator $A$ on a Hilbert space $H$ we have from \cite[Theorem 2.19]{teschl} that
$$
\min \sigma(A)=\inf_{\|\varphi\|_H=1}\langle \varphi, A \varphi \rangle_H , \qquad \max \sigma(A)=\sup_{ 
           \| \varphi\|_H=1}\langle \varphi, A\varphi \rangle_H,
$$
we see with $H=\ell^2(\Z)$ that $\sigma(-\Delta_1 +q) \subset [\inf_\Z q, \sup_\Z q +4]$.
\end{proof}

We continue with an example for the operator $-\frac{1}{d(t)}\frac{\der^2}{\der t^2}$ and a result about its spectral properties.

\begin{lemma}\label{lem:exd}
Let $d$ be a periodic two-step potential given by
\begin{align*}
    d(t)=\begin{cases}
        a, & t \in [0,T\theta), \\
        b, & t \in [T \theta,T), \\
        d(t+T), & t \in \R
    \end{cases}
\end{align*}
where $a,b,T>0$, $a \neq b$, $\theta \in (0,1)$ and $\sqrt{a}\theta=\sqrt{b}(1-\theta)$. Then, $d$ satisfies assumption \ref{as:dbdd}, 
\begin{align*}
 \sigma\left(-\frac{1}{d(t)} \frac{\der^2}{\der t^2}\right)&= \left\{ \nu_k(l): k \in \Z, l \in \left[-\frac{\pi}{T}, \frac{\pi}{T}\right] \right\} 
 =\left[\nu_0(0),\nu_0\left(\frac{\pi}{T}\right) \right] \dot \cup \dot \bigcup_{k \in \N} \left[\nu_{-k}\left(\frac{\pi}{T}\right),\nu_k\left(\frac{\pi}{T}\right) \right]
\end{align*}
and the spectral gaps are located in 
\begin{align*}
    \dot \bigcup_{k \in \N_0} \left( \nu_k \left(\frac{\pi}{T}\right), \nu_{-(k+1)} \left(\frac{\pi}{T}\right)\right)
\end{align*}  
where
\begin{align*}
    \nu_k(l)= \left(\frac{\arccos\Bigl(1+4\theta(1-\theta)(\cos(lT)-1)\Bigr)+2\pi k}{2T\sqrt{a}\theta} \right)^2 .
\end{align*} 
\end{lemma}
\begin{proof}
Clearly, by definition, $d$ satisfies assumption \ref{as:dbdd}.
It follows from Floquet--Bloch theory, see \cite{Eastham}, that the spectrum can be calculated explicitly by solving the following semi-periodic eigenvalue problem on $[0,T]$
    \begin{align*}
        -\ddot \psi(t)&=\nu d(t) \psi(t), \quad t \in (0,T) ,\\
        \psi(T)&=\ee^{\ii T l } \psi(0), \quad l \in \left[-\frac{\pi}{T}, \frac{\pi}{T}\right]
    \end{align*}
    which leads to 
    \begin{align*}
 \sigma\left(-\frac{1}{d(t)} \frac{\der^2}{\der t^2}\right)&= \left\{ \nu_k(l): k \in \Z, l \in \left[-\frac{\pi}{T}, \frac{\pi}{T}\right] \right\} 
 \end{align*}
 with $\nu_k(l)$ as given in the statement.
Next, we collect some observations on $\nu_k(l)$: For all $k \in \Z$, $\nu_k(l)$ is continuous and even in $l$ so that it suffices to consider $l\in\left[ 0, \frac{\pi}{T}\right]$. For such $l$ the values $\nu_0(l)$ and $\nu_k(l)$ are monotonically increasing in $l$ for all $k \in \N$ and $\nu_{-k}(l)$ is monotonically decreasing in $l$ for all $k \in \N$.
This implies for all $k \in \N$ that
\begin{align*}
    0=\nu_0(0) < \nu_0 \left(\frac{\pi}{T}\right) \quad \mbox{ and } \nu_{-k}\left(\frac{\pi}{T}\right) < \nu_{-k}(0)=\nu_k(0) < \nu_k\left(\frac{\pi}{T}\right).
\end{align*}
Therefore the gap lengths given by
\begin{align*}
\nu_{-(k+1)} \left(\frac{\pi}{T}\right) - \nu_k \left(\frac{\pi}{T}\right)&=\pi (2k+1)  \frac{\pi- \arccos(1-8 \theta(1-\theta))}{a T^2 \theta^2}, \quad k\in \N_0
\end{align*}
are all positive and hence  
\begin{align*}
     \sigma\left(-\frac{1}{d(t)} \frac{\der^2}{\der t^2}\right)=\left[\nu_0(0),\nu_0\left(\frac{\pi}{T}\right) \right] \dot \cup \dot\bigcup_{k \in \N} \left[\nu_{-k}\left(\frac{\pi}{T}\right),\nu_k\left(\frac{\pi}{T}\right) \right].
\end{align*}
\end{proof}

\begin{corollary} \label{cor:zitierbar} With $d$ as in Lemma~\ref{lem:exd} and $q\in \ell^\infty(\Z)$ the assumption~\ref{as:specgap} is fulfilled provided $\theta=\frac{\sqrt{b}}{\sqrt{a}+\sqrt{b}}\leq  \frac{2-\sqrt{3}}{4}$ and there is some $k\in\N_0$ such that 
$$
\left(\frac{\pi}{T\sqrt{a}\theta}\right)^2 \left(k+\frac16\right)^2< \inf_\Z q \leq  \sup_\Z q < \left(\frac{\pi}{T\sqrt{a}\theta}\right)^2 \left(k+\frac56\right)^2-4.
$$
For sufficiently large $k$ this inequality can be fulfilled by a suitable choice of $q\in \ell^\infty(\Z)$ and if additionally  $T < \frac{\pi}{\theta\sqrt{6a}}$ then it can be fulfilled for every $k\in \N_0$. 
\end{corollary}
\begin{proof} For $k\in \N_0$ we have $\nu_k(\frac{\pi}{T})= \left(\frac{\arccos\bigl(1-8\theta(1-\theta)\bigr)+2\pi k}{2T\sqrt{a}\theta} \right)^2$ and by the assumption on $\theta$ we have $1-8\theta(1-\theta)\in [\frac12,1)$ so that $\arccos\bigl(1-8\theta(1-\theta)\bigr)\in (0, \frac{\pi}{3}]$. 
Therefore, for $k\in \N_0$
$$
    \nu_k\left(\frac{\pi}{T}\right)  \leq \left(\frac{\pi}{T\sqrt{a}\theta}\right)^2 \left(k+\frac16\right)^2, \quad \nu_{-(k+1)}\left(\frac{\pi}{T}\right) \geq \left(\frac{\pi}{T\sqrt{a}\theta}\right)^2 \left(k+\frac56\right)^2.
$$
If the assumption of the corollary is fulfilled then $[\inf_\Z q, \sup_\Z q+4]\subset \left(\nu_k\left(\frac{\pi}{T}\right),\nu_{-(k+1)}\left(\frac{\pi}{T}\right)\right)$ and thus assumption~\ref{as:specgap} is fulfilled due to Lemma~\ref{lem:spec_disc_lap_q}. Moreover, the gap lengths $\nu_{-(k+1)}\left(\frac{\pi}{T}\right)-\nu_k\left(\frac{\pi}{T}\right)\to \infty$ monotonically as $k\to \infty$. The smallest gap length, which occurs for $k=0$, is larger than $4$ provided $\frac{\pi}{T\sqrt{a}\theta}> \sqrt{6}$.
\end{proof}

    \section*{Acknowledgments}
	Funded by the Deutsche Forschungsgemeinschaft (DFG, German Research Foundation) – Project-ID 258734477 – SFB 1173. The authors thank Robert Wegner (KIT) for a fruitful discussion on the Loewner ordering and Sebastian Ohrem (KIT) for valuable feedback.

	\printbibliography

\end{document}